\documentclass[12pt]{amsart}
\usepackage{amsfonts,amssymb}

\newtheorem{theorem}{Theorem}[section]
\newtheorem*{theorem*}{Theorem}
\newtheorem{lemma}{Lemma}[section]
\newtheorem{corollary}{Corollary}[section]
\newtheorem*{corollary*}{Corollary}
\newtheorem{proposition}{Proposition}[section]
\newtheorem*{proposition*}{Proposition}
\newtheorem{claim}{Claim}[section]

\theoremstyle{definition}
\newtheorem*{definition*}{Definition}
\newtheorem{definition}{Definition}[section]

\theoremstyle{remark}

\newtheorem*{notation}{Notation}

\theoremstyle{plain}

\newcommand{\Z}{{\mathbb Z}}

\newcommand{\C}{{\mathbb C}}

\newcommand{\N}{{\mathbb N}}

\newcommand{\CA}{{\mathcal A}}
\newcommand{\CC}{{\mathcal C}}
\newcommand{\CD}{{\mathcal D}}
\newcommand{\CF}{{\mathcal F}}
\newcommand{\CG}{{\mathcal G}}
\newcommand{\CB}{{\mathcal B}}

\newcommand{\ba}{\mathbf{a}}
\newcommand{\bb}{\mathbf{b}}

\newcommand{\bn}{\mathbf{n}}

\newcommand{\bx}{\mathbf{x}}
\newcommand{\by}{\mathbf{y}}

\newcommand{\bz}{\mathbf{z}}
\newcommand{\bu}{\mathbf{u}}
\newcommand{\bv}{\mathbf{v}}
\newcommand{\bw}{\mathbf{w}}

\newcommand{\QQ}{{\bf Q}}

\newcommand{\K}{{\bf K}}
 
\newcommand{\CQ}{{\bf Q}}

\newcommand{\inv}{^{-1}}
\newcommand{\id}{\bf Id}

\newcommand{\type}[1]{^{[#1]}}

\newcommand{\RP}{{\bf RP}}

\newcommand\Cubed{\{0,1\}^d}

\newcommand\one{{\bf 1}}
\newcommand{\norm}[1]{\Vert #1\Vert}
\newcommand{\nnorm}[1]{|\!|\!|#1|\!|\!|}

\newcommand{\prop}{{\mathcal P}(d)}

\begin{document}

\title
{Nilsequences and a structure theorem for topological dynamical systems}

\author{Bernard Host}
\address{Laboratoire d'analyse et de math\'ematiques appliqu\'{e}es, 
Universit\'e de Paris-Est, Marne la Vall\'ee \& CNRS UMR 8050\\
5 Bd. Descartes, Champs sur Marne\\
77454 Marne la Vall\'ee Cedex 2, France}
\email{bernard.host@univ-mlv.fr}

\author{Bryna Kra}
\address{ Department of Mathematics, Northwestern University \\ 2033 Sheridan Road Evanston \\ IL 60208-2730, USA} 
\email{kra@math.northwestern.edu}

\author{Alejandro Maass}
\address{Departamento de Ingenier\'{\i}a
Matem\'atica, Universidad de Chile
\& Centro de Modelamiento Ma\-te\-m\'a\-ti\-co 
UMI 2071 UCHILE-CNRS \\ Casilla 170/3 correo 3 \\
Santiago, Chili.}
\email{amaass@dim.uchile.cl} 



\thanks{The first author was partially supported by the Institut
Universitaire de France, the second author by NSF grant 
$0555250$, and 
the third author by the Millennium Nucleus Information and 
Randomness P04-069F, CMM-Fondap-Basal fund.  This work was begun 
during the visit of the authors to MSRI and we thank the institute 
for its hospitality.}

\begin{abstract}We characterize inverse limits of nilsystems 
in topological dynamics, via a structure theorem for topological 
dynamical systems that is an analog of the structure theorem 
for measure preserving systems. 
We provide two applications of the structure. 
The first is to nilsequences, which have played an important 
role in recent developments in ergodic theory and additive 
combinatorics; we give a characterization that detects if a given 
sequence is a nilsequence by only testing properties locally, meaning 
on finite intervals.
The second application is the construction of the maximal nilfactor 
of any order in a distal minimal topological dynamical system.  We show 
that this factor can be defined via a certain generalization of
the regionally proximal relation that is used  to produce the maximal equicontinuous factor and corresponds to the case of order $1$.  
\end{abstract}

\maketitle

\date{May 13, 2009}

 \markboth{Bryna Kra, Bernard Host, Alejandro
Maass}{Nilsequences and a structure theorem}

\section{Introduction}

\subsection{Nilsequences}
The connection between ergodic theory and additive combinatorics 
started in the 1970's, with Furstenberg's beautiful proof of Szemer\'edi's 
Theorem via ergodic theory.  Furstenberg's proof paved the way for 
new combinatorial results via ergodic methods, as well as leading to 
numerous developments within ergodic theory.  
More recently, the interaction between the 
fields has taken a new dimension, with ergodic objects being imported 
into the finite combinatorial setting.  Some objects at the center of this interchange are
nilsequences and the nilsystems on which they are defined.
They enter, for example, in ergodic theory into convergence of multiple ergodic averages~\cite{HK1} and 
into  the theory of multicorrelations~\cite{BHK}.  In number theory, they arise in finding
patterns in the primes (see ~\cite{GT} and the companion 
articles~\cite{GT2} and~\cite{GT3}).  In combinatorics, they are used to find intricate patterns in subsets 
of integers with positive upper density~\cite{FrWi}.  

Nilsequences are defined by evaluating a function along 
the orbit of a point in the homogeneous space of a nilpotent 
Lie group.  In a variety of 
situations, nilsequences have been used to test for a 
lack of uniformity of a function.  Yet, the local properties 
of nilsequences are not well understood.  
It is difficult to detect if a given sequence is a nilsequence, 
particularly if one only knows {\em local} information about the sequence, 
meaning properties that can only be tested on finite intervals.

We recall the definition of a nilsequence.  A {\em basic 
$d$-step nilsequence} is a sequence of the form 
$(f(T^nx)\colon n\in\Z)$, where  
$(X, T)$ is a $d$-step nilsystem, $f\colon X\to \C$ is a 
continuous function,  and $x\in X$. 
 A  {\em $d$-step nilsequence} is a uniform limit of basic $d$-step nilsequences.  (See Section~\ref{sec:nilsystems} for the definition 
 of a nilsystem.)
We give a characterization of nilsequences of 
all orders that can be tested locally, 
generalizing the work in~\cite{HM} that gives such an analysis for $2$-step 
nilsequences.  

We look at finite portions, the ``windows'', of a sequence and we are interested in 
finding a copy of the same finite window up to some given 
precision.  To make this clear, we introduce some notation.  
For a sequence $\ba= (a_n\colon n\in\Z)$, integers $k,j,L$, and a 
real $\delta > 0$, if each entry in the window $[k-L, k+L]$ is 
equal to the corresponding entry in the window $[j-L, j+L]$ up to an 
error of $\delta$, then we write 
\begin{equation}
\label{eq:approx}
\ba_{[k-L, k+L]} =_\delta \ba_{[j-L, j+L]}\ .
\end{equation}

The characterization of almost periodic sequences (which are exactly $1$-step 
nilsequences) by compactness can 
be formulated as follows:

\begin{proposition*}
The bounded sequence $\ba= (a_n\colon n\in\Z)$ of 
complex numbers is almost periodic if and only if 
for all $\varepsilon > 0$, there exist an integer $L\geq 1$ 
and a real $\delta > 0$ such that for any $k,n_1,n_2 \in\Z$
whenever $\ba_{[k-L, k+L]} =_\delta \ba_{[k+n_1-L, k+n_1+L]}$ and 
$\ba_{[k-L, k+L]} =_\delta \ba_{[k+n_2-L, k+n_2+L]}$ then 
$|a_k-a_{k+n_1+n_2}|< \varepsilon$.  
\end{proposition*}

We give a similar characterization for a $(d-1)$-step nilsequence $\ba$:  
if in every interval of a given length 
the translates of the sequence $\ba$ along  
finite sums (i.e. cubes) of any sequence $\bn = (n_1, \ldots, n_{d})$ 
are $\delta$-close to 
the original sequence {\em except} possibly at the 
sum $n_1+\ldots +n_{d}$, then we also 
have control over the distance between $\ba$
and the translate by $n_1+\ldots +n_{d}$. 

The general case is:
\begin{theorem}
\label{theorem:nilseq}
Let $\ba=(a_n: n\in \Z)$ be a bounded 
sequence of complex numbers and let $d\geq 2$ be an integer. The sequence 
$\ba$ is a $(d-1)$-step nilsequence if and only if  for every 
$\varepsilon>0$ there exist an integer $L\geq 1$ and real 
$\delta >0$ such that for any $(n_1,\ldots,n_{d})\in\Z^d$ and 
$k\in\Z$, whenever
$$\ba_{[k+\epsilon_1n_1 +\ldots +\epsilon_dn_d -L, k+\epsilon_1n_1 
+\ldots +\epsilon_dn_d +L]} =_\delta 
\ba_{[k-L, k+L]}$$
for all choices of $\epsilon_1, \ldots, \epsilon_d\in\{0,1\}$ other 
than $\epsilon_1= \ldots = \epsilon_d = 1$, then we have
$|a_{k+n_1+\ldots+n_d} - a_k| < \varepsilon$.  
\end{theorem}

In fact, we can replace the approximation in~\eqref{eq:approx}  in both the hypothesis and conclusion 
by any other approximation that defines pointwise convergence and have the analogous result. 

\subsection{A structure theorem for topological dynamical systems}

We prove a structure theorem for topological dynamical 
systems that gives a  characterization of 
inverse limits of nilsystems.  
Theorem~\ref{theorem:nilseq} follows from this structure theorem, 
exactly as it does in the case for $d=2$ in~\cite{HM}, where the proof 
of this implication can be found.
The structure theorem for topological dynamical systems 
can be viewed as an analog of 
the purely ergodic structure theorem of~\cite{HK1}. 
We introduce the following structure: 
\begin{definition}
\label{def:parallelepipeds}
Let $(X,T)$ be a topological dynamical 
system and let $d\geq 1$ be an integer.  
We define $\QQ\type d(X)$ to be the closure in $X^{2^d}$ 
of elements of the form
$$
(T^{n_1\epsilon_1+\ldots+n_d\epsilon_d}x\colon\epsilon = (\epsilon_{1}, \ldots, \epsilon_{d})\in\{0,1\}^d) \ , 
$$
where $\bn=(n_1, \ldots, n_d)\in\Z^d$, $x\in X$, and we denote a point of $X^{2^d}$ by 
$(x_\epsilon\colon\epsilon\in\{0,1\}^d)$. 
When there is no 
ambiguity,  we write $\QQ\type d$ instead of $\QQ\type d(X)$.
An element of $\QQ\type d(X)$ is called a {\em (dynamical) parallelepiped 
of dimension $d$}.  
\end{definition}

As examples, $\QQ\type 2$ is the closure in $X^{4}$ of the set  
$$\{(x,T^mx,T^nx,T^{n+m}x)\colon x \in X, m,n\in \Z\}$$ and 
$\QQ\type 3$ is the closure in $X^{8}$ of the set  
\begin{multline*}
\bigl\{(x,T^mx,T^nx,T^{m+n}x,T^px,T^{m+p}x,T^{n+p}x,T^{m+n+p}x)\colon \\ x 
\in X, m,n,p\in \Z\bigr\}\ .
\end{multline*}
In each of these, the indices $m,n$ and $m,n,p$ can be taken in $\N$ rather than $\Z$, giving rise to the same object.  
This is obvious if $T$ is invertible, but can also be proved without the assumption of invertibility.  
Thus, throughout the article, we assume that all maps are invertible.

We use these parallelepipeds structures to characterize 
nilsystems:

\begin{theorem}\label{theorem:main}
Assume that $(X,T)$ is a transitive topological dynamical system and let $d\geq 2$ be an integer.  
The following properties are equivalent:
\begin{enumerate}
\item  
\label{it:change-one}
If $\bx,\by\in\QQ\type {d}(X)$ have $2^{d}-1$ coordinates in common, 
then $\bx=\by$.  
\item If $x,y\in X$ are such that $(x, y, \ldots, 
y)\in\QQ\type{d}(X)$, then $x=y$.  
\item 
\label{it:nil}
$X$ is an inverse limit of $(d-1)$-step minimal nilsystems.  
\end{enumerate}
\end{theorem}
(For definitions of all the objects, see Section~\ref{section:parallele}.)
We note that the use of both $d$ and $d-1$ is necessary throughout the article, 
and this leads us to use whichever is notationally more convenient at 
various times in the proofs.

The first property clearly implies the second, since $(y, y, 
\ldots, y)\in\QQ\type{d}(X)$ for all $y\in X$.  
The second property implies that the system is distal (see Section \ref{section:parallele}).  
The second property plus the assumption of distality implies the first property 
(see Section \ref{section:paradistal}), which together give that 
the first two properties are equivalent.

Systems satisfying these properties
 play a key role in the article and so we define:
\begin{definition}
A transitive system satisfying either of the first two equivalent properties of 
Theorem~\ref{theorem:main} is called a {\em system of order $d-1$}.  
\end{definition}

The implication~\eqref{it:nil} $\Rightarrow$ \eqref{it:change-one}  in Theorem~\ref{theorem:main} follows from results in~\cite{HK3} and 
is reviewed here in Proposition~\ref{prop:Qnil}.  
The implication \eqref{it:change-one} $\Rightarrow$ \eqref{it:nil} is 
proved in Section \ref{sec:final}, using completely different methods 
from that used in~\cite{HM}  for $d=3$, and 
proceeds by introducing an invariant measure on $X$.
 
\subsection{The regionally proximal relation and generalizations}
We give a second application of Theorem~\ref{theorem:main} 
in topological dynamics. 
The study of maximal equicontinuous factors is classical (see, 
for example~\cite{Aus}). 
The maximal equicontinuous factor is the topological 
analog of the Kronecker factor in ergodic theory and recovers 
the continuous eigenvalues of a system.  
There are several ways to construct this factor, but the standard 
method is as a quotient of the {\em regionally proximal relation}.  
The first step in generalizing this relation was carried out in~\cite{HM}, 
where the concept of a double regionally  proximal relation is introduced and is 
used in the distal case to define the maximal $2$-step nilfactor. In this article we generalize this relation for 
higher levels and for $d\geq 1$ we define the {\em regionally proximal relation of order $d$}, referring to it as  $\RP \type d$. 
While these generalizations  were motivated by the study of abstract parallelepipeds in additive combinatorics~\cite{HK2}, they require new techniques.
Although we defer the definition of the regionally proximal relation of order $d$
until Section~\ref{section:parallele}, we summarize its uses.  

\begin{proposition}
\label{prop:orderRP}
Assume that $(X,T)$ is a transitive topological dynamical system and that 
$d \geq 1$ is an integer. If the regionally proximal relation of order $d$ on $X$ 
 is trivial, then the system is distal.
 \end{proposition}

In a distal system, we show that $\RP \type d$ is an equivalence relation and that
it defines the {\em maximal $d$-step topological nilfactor} of the system. 

\begin{theorem}\label{theorem:maxd}
Assume that $(X,T)$ is a distal minimal system and that  $d \geq 1$ is an integer.  
Then the regionally proximal relation  of order $d$ on $X$ is a closed invariant equivalence relation 
and the quotient of $X$ under this relation is its maximal $d$-step nilfactor. 
\end{theorem}

The maximal $d$-step (topological) nilfactor is the topological analog of the ergodic theoretic factor ${\mathcal Z}_d$ 
constructed in~\cite{HK1}.  These ergodic factors are characterized by inverse limits of $d$-step nilsystems. 
In this direction, we prove in the distal case that $\RP \type d$ is trivial if and only if the system itself is  an inverse 
limit of $d$-step nilsystems.   

To prove Theorem~\ref{theorem:maxd} we  show in Proposition~\ref{prop:quaotient1} that the quotient of $X$ under $\RP \type d$ is its maximal factor of order 
$d$.  From Theorem~\ref{theorem:main}, we deduce that notions of a system of order $d$ and an inverse limit of $d$-step nilsystems are equivalent, giving us the conclusion.  

We conjecture that the hypothesis of distality in Theorem~\ref{theorem:maxd} 
is superfluous, but were unable to prove this.  

\subsection{Guide to the paper}
The  article is divided into two somewhat distinct parts. In the first part (Sections~\ref{section:parallele} and~\ref{section:paradistal}), we develop the topological theory of parallelepipeds and the associated theory of generalized regionally proximal relations. With the topological methods developed in these sections, we are able to prove all but the implication ``(1) $\Rightarrow$ (3)'' of Theorem~\ref{theorem:main}. 
In Section~\ref{section:parallele}, we state the properties of parallelepiped structures and the relation with 
generalized regionally proximal pairs and show how the conditions of 
Theorem~\ref{theorem:main} imply that the system is distal. 
In Section~\ref{section:paradistal}, we prove that 
in the distal case, the main structural properties of parallelepipeds (the ``property of closing parallelepipeds'') allows us to show 
that first two conditions in Theorem~\ref{theorem:main} are equivalent and to show that regionally proximal relation of order $d$ gives rise to the maximal factor of order $d$.  
The proof of the remaining implication is carried out in Section~\ref{sec:final} and relies heavily on ergodic theoretic notions of Section~\ref{section:preergodic}. 
However, the interaction of the topological and measure theoretic structures plays a key role in the analysis, and it is only via measure theoretic methods that we are finally able to obtain the general topological results. 

\section{Background}

\subsection{Topological dynamical systems}

A {\em transformation} of a compact metric space $X$ is a homeomorphism
of $X$ to itself.  
A {\em topological dynamical system}, referred to more 
succinctly as just a {\em system}, is a pair 
$(X,T)$, where $X$ is a compact metric space and 
$T\colon X\to X$ is a transformation.   
We use $d_X(\cdot,\cdot)$ to denote the metric in $X$ and 
when there is no ambiguity, we write $d(\cdot, \cdot)$. 
We also make use of 
a more general definition of a topological system.
That is, instead of just a single transformation $T$, we consider 
commuting homeomorphisms $T_1,\ldots,T_k$ of $X$ or a countable abelian group of 
transformations. 
We summarize some basic definitions and properties of 
systems in the classical setting of one transformation. 
Extensions to the general case are straightforward.

A {\em factor} of a system $(X,T)$ is another system 
$(Y,S)$ such that there exists a continuous and onto map 
$p\colon X\to Y$ satisfying $S \circ p= p \circ T$.  
The map $p$ is called a {\em factor map}.  
If $p$ is bijective, the two systems are {\em (topologically) conjugate}.
In a slight abuse of notation, when there is no ambiguity, 
we denote all transformations (including ones in possibly distinct systems) by $T$. 

A system $(X,T)$ is {\em transitive} if there exists some point 
$x\in X$ whose orbit  $\{T^n x\colon n\in\Z\}$ is dense 
in $X$ and we call such a point a {\em transitive point} . The system is {\em minimal} if the orbit of any point is dense in $X$. 
This property is equivalent to saying 
that $X$ and the empty set are the only closed 
invariant sets in $X$.

\subsection{Distal Systems}
\label{sec:distal}
The system $(X,T)$ is {\em distal} if for any pair of distinct 
points $x,y \in X$, 
\begin{equation}
\label{eq:distal}
\inf_{n\in \Z} d(T^n x,T^n y) >0\ .
\end{equation}
In an arbitrary system, 
pairs satisfying property~\eqref{eq:distal} are called {\em distal pairs}. 
The points $x$ and $y$ are {\em proximal} 
if $\liminf_{n\to \infty} d(T^nx,T^ny)=0$.  

The following proposition summarizes some basic properties of 
distal systems:

\begin{proposition} (See Auslander~\cite{Aus}, chapters 5 and 7) 
\label{prop:distal}
\begin{enumerate}
\item The Cartesian product of a finite family of 
distal systems is a distal system.
\item If $(X,T)$ is a distal system and $Y$ is a closed and 
invariant subset of $X$, then 
 $(Y,T)$ is a distal system.
\item A transitive distal system is minimal.
\item A factor of a distal system is distal.
\item Let $p\colon X\to Y$ be a factor map between the distal 
systems $(X,T)$ and $(Y,T)$. If 
$(Y,T)$ is minimal, then $p$ is an open map.
\end{enumerate}
\end{proposition}

Up to the obvious changes in notation, 
this proposition holds for systems with a countable abelian group of 
transformations acting on the space $X$.

For later use, we note the following lemma on distal systems:
\begin{lemma}
\label{lem:extendcontinu}
Let $(X,T)$ and $(Y,T)$ be two minimal systems and assume that 
$(Y,T)$ is distal.  If $X_1$ is a nonempty invariant subset of $X$ and 
$\Phi\colon X_1\to Y$ is a continuous map on $X_1$ with the induced topology and 
commuting with the 
transformations $T$, then $\Phi$ 
has a continuous extension to $X$.
\end{lemma}

\begin{proof}
Let $\Gamma\subset X\times Y$ be the graph of $\Phi$:
$$
 \Gamma=\{(x,\Phi(x))\colon x\in X_1\}\ .
$$
Let $\overline\Gamma$ be 
the closure of $\Gamma$ in $X\times Y$. 
We claim that $\overline\Gamma$ is the graph of some map $\Phi'\colon 
X\to Y$.

The projection of $\overline\Gamma$ on $X$ is a closed invariant 
subset of $X$ containing $X_1$, and by minimality this projection is 
equal to $X$. 
Assume that $x\in X$ and $y,y'\in Y$ are such that $(x,y)$ and 
$(x,y')$ belong to $\overline\Gamma$.  Let $x_1\in X_1$ and chose a 
sequence $(n_i)_{i\in\N}$ of integers such that 
$T^{n_i}x\to x_1$ and such that the 
sequences $(T^{n_i}y)_{i\in\N}$ and  $(T^{n_i}y')_{i\in\N}$ 
converge in $Y$, to the 
points $z$ and $z'$, respectively, as $i\to\infty$. 
Then $(x_1,z)$ and $(x_1,z')$ belong to 
$\overline\Gamma\cap(X_1\times Y)$. 

On the other hand, since $\Phi$ is continuous on $X_1$, we have that 
$\overline\Gamma\cap(X_1\times Y)=\Gamma$ and thus $z=\Phi(x_1)=z'$. 
Since $(Y,T)$ is distal, 
we conclude that $y=y'$ and we have that 
$\overline{\Gamma}$ is the graph of a map $\Phi'\colon X\to Y$.  

The restriction of $\Phi'$ to $X_1$ is equal to $\Phi$ and 
because its graph is closed, $\Phi'$ is continuous. 
Finally, since $X_1$ is invariant and nonempty, 
it is dense in $X$.  By
minimality and density, we conclude
that $\Phi'\circ T=T\circ\Phi'$.
\end{proof}

\subsection{Nilsystems and nilsequences}
\label{sec:nilsystems}

\begin{definition}  
Let $d\geq 1$ be an integer and assume that $G$ is a 
$d$-step nilpotent Lie group and that 
$\Gamma\subset G$ is a discrete, cocompact subgroup of $G$. 
The compact manifold $X = G/\Gamma$ is a {\em $d$-step 
nilmanifold} and $G$ acts naturally on $X$ 
by left translations:  $x\mapsto \tau.x$ for $\tau \in G$. 

If $T$ is left multiplication on $X$ by some fixed element 
of $G$, then $(X,T)$ is called a {\em $d$-step nilsystem}.  
\end{definition}

A $d$-step nilsystem is an example of a 
distal system.
In particular if the nilsystem is transitive, then it is minimal.
Also, the closed orbit of a point in a 
$d$-step nilsystem is topologically conjugate to a $d$-step nilsystem. 
See~\cite{AGH}, \cite{Pa}, and~\cite{Le} for proofs 
and general references on nilsystems.  

We also make use of inverse limits of nilsystems and so we recall 
the definition of an inverse limit of systems (restricting ourselves to the case of 
sequential inverse limits).  
If $(X_i, T_i)_{i\in\N}$ are systems and $\pi_{i}\colon X_{i+1}\to X_i$ 
are factor maps, the {\em inverse limit} of the systems 
is defined to be the compact subset of $\prod_{i\in\N} X_i$ given by 
$$
\{(x_i)_{i\in\N}\colon\pi_i(x_{i+1}) = x_i\} \ .
$$
It is a compact metric space endowed with the 
distance 
$$
d(x,y) = \sum_{i\in\N} 1/2^id_i(x_i,y_i) \ .
$$
We note that the maps $T_i$ induce a transformation $T$ on the 
inverse limit.

Many properties of the 
systems $(X_i, T_i)$ also pass to the inverse limit, including 
minimality, distality, and unique ergodicity.

We return to the definition of a nilsequence:
\begin{definition}  
If $(X = G/\Gamma, T)$  is a $d$-step nilsystem, where $T$ is given by multiplication by the element $\tau\in G$, 
$f\colon X\to \C$ is a  continuous function, and $x\in X$, the sequence 
$(f(\tau^n.x)\colon n\in\Z)$ is a {\em basic $d$-step nilsequence}. 
A uniform limit of basic $d$-step nilsequences is a {\em nilsequence}.  

Equivalently, 
a $d$-step nilsequence is given by  $(f(T^nx)\colon n\in\Z)$, 
where  $(X, T)$ is an inverse limit of $d$-step nilsystems, 
$f\colon X\to \C$ is a continuous function and $x\in X$.
\end{definition}  

The two statements in the definition are shown to be equivalent in
Lemma 14 in~\cite{HM}.
Moreover, in the definition of a $d$-step nilsequence, 
we can assume that the system is minimal. Namely, considering
the closed orbit of $x_0$, this is a transitive and so 
minimal system. 

The $1$-step nilsystems are translations on compact 
abelian Lie groups and $1$-step nilsequences are exactly almost 
periodic sequences (see~\cite{Pa}). Examples of $2$-step nilsequences and a 
detailed study of them are given in~\cite{HK4}.

\section{Dynamical Parallelepipeds: first properties}\label{section:parallele}

\subsection{Notation}

Let $X$ be a set, let $d\geq 1$ an integer, 
and write $[d]=\{1,2,\dots,d\}$. 
We view $\{0,1\}^d$ in one of two ways, either as a sequence 
$\epsilon =\epsilon_1\ldots\epsilon_d$ of 
$0$'s and $1$'s written without commas or parentheses; 
or as a subset of $[d]$.  
A subset $\epsilon$ corresponds to the sequence 
$(\epsilon_1,\ldots,\epsilon_d) \in \Cubed$ 
such that $i \in \epsilon$ if and only if 
$\epsilon_i=1$ for $i \in [d]$.  

If $\bn=(n_1, \ldots, n_d)\in\Z^d$ and $\epsilon\subset[d]$, 
we define 
$$
\bn\cdot\epsilon = \sum_{i=1}^dn_i\epsilon_i  = 
\sum_{i \in \epsilon} n_i \ .
$$

We denote $X^{2^d}$ by $X\type d$. A point 
$\bx\in X\type{d}$ can be  written in one of 
two equivalent ways, depending on the context:  
$$\bx=(x_\epsilon\colon\epsilon \in \Cubed) = (x_\epsilon\colon\epsilon\subset [d]) \ .$$

For $x\in X$, we write $x\type{d} = (x, x, \ldots, x)\in X\type d$.
The diagonal of  $X\type d$ is $\Delta\type{d}=\{x\type d \colon x \in X\}$.

A point $\bx \in X\type{d}$ can be decomposed as $\bx=(\bx',\bx'')$ with 
$\bx',\bx'' \in X\type{d-1}$, where $\bx' = (x_{\epsilon 0}\colon 
\epsilon \in \{0,1\}^{d-1})$ and  $\bx'' = (x_{\epsilon 1}\colon 
\epsilon \in \{0,1\}^{d-1})$.
We can also isolate the first coordinate, writing 
$X_*\type d=X^{2^d-1}$ and then writing a point $\bx\in X\type d$ as 
$\bx = (x, \bx_*)$, where $\bx_* = 
(x_\epsilon\colon\epsilon\neq\emptyset)\in X_*\type d$. 

The {\em faces} of dimension $r$ of a point in $\bx \in X\type d$ are defined as follows.
Let  $J\subset [d]$ with $|J|=d-r$ and $\xi \in \{0,1\}^{d-r}$.  
The elements $(x_\epsilon \colon \epsilon \in\{0,1\}^d , \  \epsilon_J=\xi)$  of $X\type{r}$ 
are called {\em faces } of dimension $r$ of $\bx$, where $ \epsilon_J=(\epsilon_i\colon i \in J)$. 
Thus any face of dimension $r$ defines a natural projection from 
$X\type d$ to $X\type r$, and we call this the {\em projection along 
this face}. 

Identifying $\{0,1\}^d$ with the set of vertices of the Euclidean 
unit cube, a Euclidean isometry of the unit cube 
permutes the vertices of the cube and thus the coordinates 
of a point $x\in X\type d$.  These permutations are 
the {\em Euclidean permutations} of $X\type d$.
Examples of Euclidean permutations are permutations of digits, 
meaning a permutation of $\{0,1\}^d$ induced by a permutation of 
$[d]$, and symmetries, such as replacing $\epsilon_i$ by 
$1-\epsilon_i$ for some $i$.  For $d=2$, an example of a  digit 
permutation is the map $(00, 01, 10, 
11)\mapsto (00, 10, 01, 11)$ and an example of a symmetry is the map $(00, 
01, 10, 11)\mapsto (01, 00, 11, 10)$.  

\subsection{Dynamical parallelepipeds}
\label{subsec:dynparallel}

We recall that $\QQ\type d$ is the closure in $X^{2^d}$ 
of elements of the form
$$
(T^{n_1\epsilon_1+\ldots+n_d\epsilon_d}x\colon\epsilon\in\{0,1\}^d) \ , 
$$
where $\bn=(n_1, \ldots, n_d)\in\Z^d$  and $x\in X$ (Definition \ref{def:parallelepipeds}). 
It follows immediately from the definition that $\QQ\type d$ contains the diagonal. 

Some other basic structural properties of $\QQ\type d$ are: 
\begin{enumerate}
\item
\label{item:faces}
Any face of dimension $r$ of any $\bx\in\QQ\type {d}$ 
belongs to $\QQ\type{r}$.  (This condition is trivial for $d=2$.)
\item
\label{item:Euclidean}
$\QQ\type d$ is 
invariant under the Euclidean permutations of $X\type d$.  
\item
\label{item:QdtoQd+1}
If $\bx\in\QQ\type d$, then $(\bx,\bx)\in\QQ\type{d+1}$. 
\end{enumerate}

\begin{lemma}\label{lemma:Qfactors}
Let $d\geq 1$ be an integer, $(X,T)$ and $(Y,T)$ be systems, and 
$\pi:X\to Y$ be a factor map. Then 
$\QQ\type d(Y)$ is the image of $\QQ\type d(X)$ under 
the map $\pi^{[d]}:=\pi \times \ldots \times \pi$ ($2^d$ times). 
\end{lemma}

We can rephrase the definition of $\QQ\type d$ using some groups of 
transformations on $X\type d$.  We define:

\begin{definition}
Let  $(X,T)$ be a system and  $d\geq 1$  be an integer.  
The \emph{diagonal transformation} of $X\type d$ is the map given
 by 
$(T\type d\bx)_\epsilon = Tx_\epsilon$ for every $\bx\in X\type d$ 
and every
$\epsilon\subset[d]$.

For $ j\in [d]$, the {\em face transformation} $T_j\type d: X \type d \to X\type d$ 
is defined for every $\bx \in X\type d$ and $\epsilon \subset [d]$ by:
$$
T_j\type d\bx = 
\begin{cases}
(T_j\type d\bx)_\epsilon = Tx_\epsilon & \text{ if } j \in \epsilon \\
(T_j\type d\bx)_\epsilon = x_\epsilon  & \text{ if } j \notin \epsilon\ .
\end{cases}
$$
The \emph{face group of dimension $d$} is the group $\CF\type d(X)$ 
of transformations of $X\type d$ spanned by the face transformations.
The \emph{parallelepiped group of dimension $d$} is the group
$\CG\type d(X)$ spanned by the diagonal transformation and the face 
transformations.
We often write $\CF\type d$ and $\CG\type d$ instead of $\CF\type 
d(X)$ and $\CG\type d(X)$, respectively.  
For $\CG\type d$ and $\CF\type d$, we use similar notations to that used for $X\type d$: namely, 
an element of either of these groups is written as $S = (S_\epsilon\colon \epsilon\in\{0,1\}^d)$.  
In particular, $\CF\type d = \{S\in\CG\type d\colon S_\emptyset = \id\}$.  
\end{definition}

We note that the group $\CG\type d$ satisfies the three properties~\eqref{item:faces}--~\eqref{item:QdtoQd+1} above, 
with  $\QQ\type d$  replaced by  $\CG\type d$.  
Moreover, for $S\in\CF\type d$, we have that 
$(S,S)\in\CF\type{d+1}$.  As well, 
$\CF\type d$ is invariant under digit permutations.  

The following lemma follows directly from the definitions:

\begin{lemma}
\label{lemma:no-name}
Let $(X,T)$ be a  system and let $d\geq 1$ be an integer.  
Then $\QQ\type d$ is the closure in $X\type d$ of
$$
\{S x\type d\colon S \in\CF\type d, x\in X\}\ .
$$

If $x$ is a transitive point of $X$, then  $\QQ\type d$ is the closed orbit of
$x\type d$ under the group $\CG\type d$.
\end{lemma}

\subsection{Definition of the regionally proximal relations}

In this section, we discuss the relation $\RP\type d$ 
and its relation to $\QQ\type{d+1}$.

\begin{definition}\label{def:RP}
Let $(X,T)$ be a system and let $d\geq 1$ be an integer.  
The points $x,y\in X$ are said to be {\em regionally 
proximal of order $d$} if for any $\delta > 0$, there exist $x', y'\in X$ 
and a vector $\bn = (n_1, \ldots, n_d)\in \Z^d$ such that
$d(x,x')< \delta$, $d(y, y')<\delta$, and
$$
d(T^{\bn\cdot\epsilon}x', T^{\bn\cdot\epsilon}y') 
<\delta \text{ for any nonempty } \epsilon\subset [d] .
$$
(In other words, there exists $S\in\CF\type d$ such that 
$d(S_\epsilon\cdot x', S_\epsilon\cdot y')<\delta$ for every 
$\epsilon\neq\emptyset$.)
We call this the {\em regionally 
proximal relation of order $d$} and denote the set of regionally proximal 
points by $\RP\type d$ (or by $\RP\type d(X)$ in case of ambiguity).

Since $\RP\type{d+1}$ 
is finer than $\RP\type d$, we have defined a nested 
sequence of closed and invariant relations. 
\end{definition}

\begin{lemma}
\label{lemma:RPandQQ} 
Assume that $(X,T)$ is a transitive system 
and that $d\geq 1$ is an integer.  Then 
$(x,y)\in\RP\type d$ if and only if there exists $\ba_* \in X_*\type d$ such that 
$$(x, \ba_*, y, \ba_*) \in \CQ\type {d+1}$$
\end{lemma}

\begin{proof}
Assume that $(x,y)\in\RP\type d$.  Let $\delta>0$
and let $x',y'$ and $S$ be as in the  definition 
of regionally proximal points.
As  transitive points are dense in $X$, there exists a transitive 
point $z$ with $d(z,x')<\delta$ and, 
for every 
$\epsilon\neq\emptyset$,
$d(S_\epsilon\cdot z, S_\epsilon\cdot x')<\delta$.
 There exists an integer $k$ such that 
$d(T^k z,y')<\delta$ and that, for every $\epsilon\neq\emptyset$,
$d(S_\epsilon \cdot T^k z,S_\epsilon \cdot y')<\delta$.
We have that $d(z,x)<2\delta$, $d(T^k z,y)<2\delta$ and 
$d(S_\epsilon\cdot \ T^k z, S_\epsilon\cdot \  z)<3\delta$.

Define $\bz\in X\type{d+1}$ by
$z_{\epsilon 0} = S_\epsilon\cdot z$ and 
$z_{\epsilon 1} = S_\epsilon\cdot T^k z$ for $\epsilon\in\{0,1\}^d$.
Then $\bz=(S,S)(T\type{d+1}_{d+1})^kz\type{d+1}$ and thus 
this point belongs to $\QQ\type{d+1}$. We have that 
$d(z_\emptyset,x)<2\delta$, $d(z_{00\ldots 01},y)<2\delta$ and 
$d(z_{\epsilon 0},z_{\epsilon 1})<3\delta$ for every 
$\epsilon\in\{0,1\}^d$ different from $\emptyset$.
Letting $\delta\to 0$ and passing to a subsequence, we have a point 
of $\QQ\type{d+1}$ of the announced form.  

Conversely, if $(x,\ba_*,y,\ba_*) \in\QQ\type{d+1}$ with $\ba_* \in 
X_*\type{d}$, 
then for every $\delta > 0$, there exist $z\in X$, $\bn\in\Z^{d}$, and $p\in\Z$ 
such that $d(z,x)<\delta$, $d(T^pz,y)< \delta$, and 
$d(T^{\bn\cdot \epsilon}z, a_\epsilon) < \delta$ and
$d(T^{\bn\cdot\epsilon + p}z, a_\epsilon)< \delta$ for every nonempty 
$\epsilon\subset [d]$.  
Thus $(x,y)\in\RP\type d$.
\end{proof}

\begin{corollary}
\label{lemma:stupid}
Assume that $(X,T)$ is a transitive system 
and that $d\geq 1$ is an integer. The relation $\RP\type d(X)$ is a closed, symmetric relation that is invariant 
under $T$.

If $\phi\colon X\to Y$ is a factor map and 
if $(x, y)\in\RP\type{d}(X)$, then 
$(\phi(x), \phi(y))\in\RP\type{d}(Y)$.  
\end{corollary}

\begin{proof}
This follows immediately from the definition and Lemma~\ref{lemma:RPandQQ}.
\end{proof}

If the first property of Theorem~\ref{theorem:main} holds, then the 
relation $\RP\type d$ is trivial: if $(x,\ba_*, y, 
\ba_*)\in\QQ\type{d+1}$, then $(x, \ba_*)\in\QQ\type d$ and so 
$(x, \ba_*, x, \ba_*)\in\QQ\type{d+1}$.  By the first property of 
Theorem~\ref{theorem:main}, $x=y$.

\subsection{Reduction to the distal case}
\label{sec:reduce}

We show that systems verifying the conditions of 
Theorem~\ref{theorem:main} are distal.  

\begin{proposition}
\label{prop:rpsequal}
Assume $(X,T)$ is a transitive system  and  that $d\geq 1$ is an integer. 
If $x$ and $y$ are proximal and the closed orbit of $y$ is a minimal  
set, then 
$(x,y,y, \ldots, y)\in\QQ\type d$.
\end{proposition}

\begin{proof}
First we claim that for every $\eta > 0$, there exists $n\in\N$ such 
that $d(T^nx,y) < \eta$ and $d(T^ny,y) < \eta$. 
Since $x$ and $y$ are proximal, there exists a sequence $(m_i\colon 
i\geq 1)$ and a point $z\in X$ such that $T^{m_i}x\to z$ and 
$T^{m_i}y\to z$.  We have that $z$ belongs to the closed orbit of 
$y$, which is minimal, and so $y$ belongs to the closed orbit of 
$z$.  Thus there exists $p$ such that $d(T^pz, y)< \eta/2$.  
By continuity of $T^p$, for $i$ sufficiently large we have that 
$d(T^{m_i+p}x, y) < \eta$ and 
$d(T^{m_i+p}y, y) < \eta$.  
Setting $n = m_i+p$ for some sufficiently large $i$, we have $n$ that 
satisfies the claim.

Fix $\delta > 0$.  Applying the claim for $\eta = \delta/d$, we find some $n_1$ such that 
$d(T^{n_1}x, y) < \delta/d$ and $d(T^{n_1}y, y) < \delta/d$. 

Taking $\eta$ with $0< \eta < \delta/d$ such that $d(T^{n_1}u,T^{n_1}v) \leq \delta/d$ when
$d(u,v)\leq \eta$, and then taking $n_2$ associated to this $\eta$, from the claim we have that: 
$d(T^{n_1\epsilon_1 + n_2\epsilon_2}x, y) < 2\delta/d$ and 
$d(T^{n_1\epsilon_1 + n_2\epsilon_2}y, y) < 2\delta/d$ 
for all $\epsilon_1, \epsilon_2\in\{0,1\}^2$ other than 
$\epsilon_1=\epsilon_2 =0$.  

Thus by induction, there is a sequence of integers $n_1,  \ldots, 
n_d$ such that 
$d(T^{\bn\cdot\epsilon}x, y) < \delta$
for all $\emptyset\neq\epsilon\subset[d]$ . 
Taking $\delta\to 0$, we have the statement of the proposition.  
\end{proof}

\begin{corollary}
\label{cor:two_distal}
Assume that $(X,T)$ is a transitive system.  
If the second property of Theorem~\ref{theorem:main} holds, then 
$X$ is distal.
\end{corollary}

\begin{proof}
We first show that any point in $X$ is minimal, i.e. its closed orbit is minimal, 
and so the system is minimal.
Every $x\in X$ is proximal to some minimal point $y$ (see \cite{Aus}).  By the previous 
proposition and the hypothesis, $x=y$ and so $x$ is a minimal point. 
Applying the proposition 
to any pair of proximal points, the statement follows.
\end{proof}

\section{Parallelepipeds in distal systems}\label{section:paradistal}

\subsection{Minimal distal systems and parallelepiped structures}
\begin{lemma}
\label{lemma:Qminimal}
Let $(X,T)$ be a minimal distal system and let $d\geq 1$ be an integer. 
Then $(\QQ\type d, \CG\type d)$ is a minimal distal system.
\end{lemma}
\begin{proof}
Since $(X,T)$ is distal, so is the system 
$(X\type d, \CG\type d)$.  
Since $\QQ\type d$ is a closed and invariant subset 
of $X\type d$ under the face transformations, 
the system $(\QQ\type d, \CG\type d)$ 
is also distal.  By the second part of 
Lemma~\ref{lemma:no-name}, the system is transitive and thus 
is minimal. 
\end{proof}

Using the Ellis semigroup, Eli Glasner~\cite{glasner} showed us a proof that this lemma holds without 
the assumption of distality.

\begin{proposition}
\label{prop:equiv}
Let $(X,T)$ be a minimal distal system and let $d\geq 1$ be an integer. 
The relation $\sim_{d-1}$ defined on 
$\QQ\type{d-1}$ by
$$\bx\sim_{d-1} \bx' \text{ if and only 
if the element } (\bx, \bx')\in X\type d  \text{ belongs to }
\QQ\type d
$$ 
is an equivalence relation.  
\end{proposition}
\begin{proof}
By Property~\eqref{item:Euclidean} of Section~\ref{subsec:dynparallel}, we have that the relation is 
symmetric and by Property~\eqref{item:QdtoQd+1}, it is reflexive.  We are left 
with showing that the relation is transitive. 
Let $\bu, \bv, \bw\in\QQ\type{d-1}$ and assume that 
$(\bu,\bv)\in\QQ\type d$ and $(\bv, \bw)\in\QQ\type d$.

Choose $z\in X$.  
By Lemma~\ref{lemma:Qminimal}, the system 
$(\QQ\type d, \CG\type d)$ is minimal and so it is the closed orbit 
of $z\type d$
under the group $\CG\type d$.   
There exists a sequence $(S_{i}\colon i\geq 1)$ such that 
$S_{i}(\bu, \bv)\to z\type d=(z\type{d-1}, z\type{d-1})$ as $i\to \infty$.  
Writing $S_i = (S'_i, S''_i)$ with $S'_i, S''_i\in \CG\type{d-1}$, we have that 
$S'_i\bu\to z\type{d-1}$ and $S''_i\bv\to z\type{d-1}$.

Passing to a subsequence if needed, we can assume that 
$S''_i\bw$ 
converges to some point $\hat{\bz}\in X\type{d-1}$ as $i\to\infty$.  
We have that 
$$(S''_i, S''_i)(\bv,\bw)\to (z\type{d-1}, \hat{\bz})\in X\type d\ .
$$
But for each $i\in\N$, 
$(S''_i, S''_i)\in \CG\type d$ and thus 
$(z\type{d-1}, \hat{\bz})$ belongs to the closed orbit of $(\bv,\bw)$ under 
$\CG\type d$ and so $(z\type{d-1}, \hat{\bz})\in\QQ\type d$.

On the other hand, $S_i(\bu, \bw) = (S'_i\bu, S''_i\bw)$ converges to 
$(z\type{d-1}, \hat{\bz})$  and this point belongs to the closed orbit of $(\bu, \bw)$ under 
$\CG\type d$.  
By distality this orbit is minimal and so it follows 
that $(\bu, \bw)$ also belongs to the orbit 
closure of $(z\type{d-1}, \hat{\bz})$.  
In particular, $(\bu, \bw)\in\QQ\type d$ and the relation
$\sim_{d-1}$ is transitive. 
\end{proof}

\begin{corollary}
\label{cor:trans}
Let $(X,T)$ be a minimal distal system and let $d\geq 1$ be an integer. 
If $\bx,\by\in\QQ\type{d+1}$ and $x_\epsilon=y_\epsilon$ for all 
$\epsilon\neq\emptyset$, then $(x_\emptyset,y_\emptyset)\in\RP\type d$.
\end{corollary}
\begin{proof}
We write $\bx=(x_\emptyset, \ba_*,\bz)$ with $\ba_*\in X\type d_*$ 
and $\bz\in\QQ\type d$.  By hypothesis, $\by=(y_\emptyset,\ba_*,\bz)$ and by 
transitivity of relation $\sim_{d+1}$, we have that
$(x_\emptyset,\ba_*,y_\emptyset,\ba_*)\in\QQ\type{d+1}$.
We conclude via Lemma~\ref{lemma:RPandQQ}. 
\end{proof}
\subsection{Completing parallelepipeds}

\begin{notation}
For $x\in X$ and $d\geq 1$, write
$$
\QQ\type d(x)=\{\by\in\QQ\type d\colon y_\emptyset =x\}\ .
$$
\end{notation}

In this section, we show:
\begin{proposition}
\label{prop:QK}
For $x\in X$ and $d\geq 1$, $\QQ\type d(x)$ is the closed orbit of 
$x\type d$ under the action of the group $\CF\type d$.
\end{proposition}
Proposition~\ref{prop:QK} follows from the more general 
Proposition~\ref{prop:important} below. 

In this section (and only in this section), 
we make use of yet another notation for the points of $X\type d$:

\begin{notation}
For $\epsilon \subset [d]$, define
$$
\sigma_d(\epsilon) = \sum_{k=1}^d\epsilon_k2^{k-1} \ .
$$
For $0\leq j < 2^d$, set
$$
E(d,j) = 
\{\epsilon\subset[d]\colon \sigma_d(\epsilon)\leq j\} \ .
$$

For $x\in X$ and $d\geq 1$, let $\K\type d(x)$ denote the closed 
orbit of $x\type d$ under $\CF\type d$.
\end{notation}

We remark that $\K\type d(x)$ is minimal under the action of 
$\CF\type d$. Moreover, if $d\geq 2$ and $\by\in\K\type{d-1}(x)$, then 
$(\by,\by)\in\K\type d(x)$.  As well, $\K\type d(x)$ is invariant under 
digit permutations.  

\begin{proposition}
\label{prop:important}
Assume that $d\geq 1$ is an integer and let $0\leq j< 2^d$.  
Assume that $\bx\in X\type{d}$ satisfies the hypothesis
\begin{itemize}
\item[$H(d,j)$:]
for every $r$ and every face $F$ of dimension $r$ of $\{0,1\}^{d}$
included in $E(d,j)$, 
the projection of $\bx$ along $F$ belongs to $\CQ\type r$. 
\end{itemize}
Then there exists $\bw\in \K\type{d}(x_\emptyset)$ such that 
$w_\epsilon=x_\epsilon$ for every $\epsilon\in E(d,j)$.  
\end{proposition}

First we remark that this proposition 
implies~Proposition~\ref{prop:QK}. Indeed, if $\bx\in\QQ\type d(x)$ 
then $x_\emptyset=x$. Moreover, $\bx$ satisfies the hypothesis 
$H(d,2^d-1)$ and thus agrees with 
a point of $\K\type d(x)$ on $E(d,2^d-1)=[d]$.

\begin{proof}
For $d=1$, the result is obvious since $\K\type 
1(x_\emptyset)=\{x_\emptyset\}\times X$. For $d>1$ and $j=0$, there is 
nothing to prove.

We proceed by induction: take $d>1$ and $j>0$ and assume that 
the result holds for $d-1$ and all values of $j$ and for $d$ and 
$j'<j$.

Assume that $\bx\in  X\type{d}$ satisfies the hypothesis 
$H(d,j)$ and write $x=x_\emptyset$.

\subsubsection{} 
We first make a reduction.  We assume that the result holds under the 
additional hypothesis
\begin{itemize}
\item[(*)] $\bx$ is of the form $x_\epsilon = x_\emptyset$ for 
$\epsilon\in E(d, j-1)$
\end{itemize}
and we show that it holds in the general case.

Assume that $\bx$ satisfies $H(d,j)$.
By the induction hypothesis, there exists $\bv\in \K\type{d}(x)$ 
such that $v_\epsilon=x_\epsilon$ for all $\epsilon\in E(d,j-1)$.  By minimality, the point 
$x\type d$ lies in the closed $\CF\type{d}$-orbit of $\bv$, meaning that 
there exists a sequence $(S_\ell\colon \ell\geq 1)$ in $\CF\type{d}$ such 
that $S_\ell\bv\to x\type d$.  Passing to a subsequence, we can 
assume that $S_\ell\bx\to\bx'$.  We have that $x_\epsilon' = x$ for 
all $\epsilon\in E(d, j-1)$ and $\bx'$ satisfies property $(*)$.  

Property $H(d,j)$ is invariant under the action of $\CF\type{d}$ and 
under passage to limits.  Thus since $\bx'$ lies in the closed 
$\CF\type{d}$-orbit of $\bx$, $\bx'$ satisfies $H(d,j)$.
Using the result of the proposition with the additional assumption of 
$(*)$, we have that there exists $\bv'\in \K\type{d}(x)$ such that 
$v_\epsilon'=x'_\epsilon$ for $\epsilon\in E(d,j)$.  

Since the system is distal and $\bx'$ belongs to the closed 
$\CF\type{d}$-orbit of $\bx$, we also have that $\bx$ belongs to the closed 
$\CF\type{d}$-orbit of $\bx'$.  There exists a sequence $(S_\ell'\colon \ell\geq 
1)$ such that $S_\ell'\bx'\to\bx$.  Passing to a subsequence, we have 
that $S_\ell'\bv'\to\bu$.  Thus $\bu\in \K\type{d}(x)$ and 
$u_\epsilon=x_\epsilon$ for $\epsilon\in E(d,j)$.  

\subsubsection{} 

We now assume $\bx$ satisfies $H(d,j)$ and $(*)$ and assume that $j\neq 2^{d}-1$.  
Again, we write $x = x_\emptyset$.  

Let $\eta\in\{0,1\}^{d}$  be defined by $\sigma_{d}(\eta) = j$.  
By hypothesis, there exists some $k$  with $1\leq k 
\leq d$ such that $\eta_k = 0$.  
Choose $k$ to be the largest $k$ with this property.  

Define the map $\Phi\colon\{0,1\}^{d-1}\to\{0,1\}^{d}$ by 
$$
\Phi(\epsilon) = 
\epsilon_1\ldots\epsilon_{k-1}0\epsilon_{k}\ldots\epsilon_{d-1} \ .
$$

Setting 
$$
\theta = \eta_1\ldots\eta_{k-1}1\ldots 1\in\{0,1\}^{d-1}\ ,
$$
we have that $\Phi(\theta) = \eta$.  

Set $i = \sigma_{d-1}(\theta)$.  
It is easy to check that for $\alpha\in\{0,1\}^{d-1}$, 
\begin{equation}
\label{eq:iff}
\sigma_{d-1}(\alpha) < i \text{ if and only if } 
\sigma_{d}(\Phi(\alpha)) < \sigma_{d}(\Phi(\theta) )= j \ .  
\end{equation}
In particular, $\Phi(E(d-1,i))\subset E(d, j)$.  

Define 
$\bu\in X\type {d-1}$ to be the projection of $\bx$ on $X\type{d-1}$ along 
the face defined by $\epsilon_k = 0$.  In other words, 
$$
u_\epsilon = x_{\Phi(\epsilon)}\ , \epsilon\in\{0,1\}^{d-1}.
$$ 

Moreover, if $F$ is a face of $\{0,1\}^{d-1}$, then $\Phi(F)$ is a face 
of $\{0,1\}^{d}$.  
Since $\bx$ satisfies $H(d,j)$, we have that $\bu$ satisfies 
$H(d-1, i)$.

We have that $u_\emptyset = x$ and by the induction hypothesis, there 
exists $\bv\in \K\type {d-1}(x)$ with $v_\epsilon = u_\epsilon$ for all 
$\epsilon\in E(d-1,i)$.

Define the map $\Psi\colon\{0,1\}^{d}
\to \{0,1\}^{d-1}$ by 
$$
\Psi(\epsilon) = 
\epsilon_1\ldots\epsilon_{k-1}\epsilon_{k+1}\ldots\epsilon_{d} \ .
$$
By definition, $\Psi\circ\Phi$ is the identity 
and $\Psi(\eta) = \theta$. 
On the other hand, $\Phi\circ\Psi(\epsilon) = 
\epsilon_1\ldots\epsilon_{k-1}0\epsilon_{k+1}\ldots\epsilon_d$.  In particular, 
$$
\sigma_d(\Phi\circ\Psi(\epsilon)) \leq \sigma_d(\epsilon) 
\text{ for every } \epsilon\in\{0,1\}^d \ .
$$

Define $\bw\in X\type{d}$ by $w_\epsilon = v_{\Psi(\epsilon)}$ for 
$\epsilon\in\{0,1\}^{d}$.
In other words, $\bw$ is obtained by duplicating $\bv$ on 
two opposite faces.
We check that $\bw\in \K\type{d}(x)$.

To see this, let $\bv'$ be obtained from $\bv$ by the digit permutation 
that exchanges the digits $k-1$ and $d-1$.  Then $\bv'\in\K\type{d-1}(x)$ and 
so $(\bv',\bv')\in \K\type d(x)$.  We obtain $\bw$ from the point 
$(\bv',\bv')$ by the digit permutation that exchanges the digits $k$ and $d$. 

We claim that $\Psi(E(d,j-1))\subset E(d-1,i-1)$.
To show this, we take $\epsilon\in E(d, j-1)$ and 
distinguish two cases.  First assume there exists 
some $m$ with $k+1\leq m \leq d$ with $\epsilon_m=0$.  Then 
one of the $d-k$ last coordinates of 
$\Psi(\epsilon) = 0$ and by definition of $\theta$, 
$\sigma_{d-1}(\Psi(\epsilon)) < \sigma_{d-1}(\theta) = i$.

Now assume that here is no such $m$.  
Because $\sigma_{d}(\epsilon)< \sigma_{d}(\eta)$ and 
$\eta_k=0$, we have that $\epsilon_k = 0$.  
Then $\Phi(\Psi(\epsilon)) = \epsilon $.  
Thus
$$
\sigma_{d}(\Phi(\Psi(\epsilon))) = 
\sigma_{d}(\epsilon)  < j
$$
and applying~\eqref{eq:iff} with $\alpha = \Psi(\epsilon)$, we have that 
$\sigma_{d-1}(\Psi(\epsilon))< i$.  This proves the claim.  

We check that $\bw$ satisfies the conclusion of the proposition.  
First for $w_\eta$, we have that $w_\eta = v_{\Psi(\eta)} = v_{\theta} = 
u_{\theta}$ since $\theta\in E(d-1, i)$, and $u_\theta = x_{\Phi(\theta)} = 
x_\eta$.  Thus $w_\eta = x_\eta$.  
Next, if $\epsilon\in E(d,j-1)$, then $x_\epsilon = x$.  
On the other hand, $w_\epsilon = v_{\Psi(\epsilon)} = u_{\Psi(\epsilon)}$, 
where the last equality holds because by the claim we have
$\Psi(\epsilon) \in 
E(d-1,i-1)$.  But 
$u_{\Psi(\epsilon)} = x_{\Phi\circ\Psi(\epsilon)} = x$, because
$\sigma_d(\Phi\circ\Psi(\epsilon)) \leq \sigma_d(\epsilon)\leq j-1$.  
This $\bw$ is as announced.  

\subsubsection{} 

We are left with considering the case that $j=2^{d}-1$
The hypothesis $H(d, 2^{d}-1)$ means that 
$\bx=(x, x, \ldots, x, y)\in \CQ\type{d}$ and we have to show that this 
lies in $\K\type{d}(x)$. 

We start with a general property. Writing a point $\bx\in X\type{d}$ as $\bx = (\bx',\bx'')$, define 
the projection $\phi\colon \K\type{d}(x)\to\CQ\type {d-1}$
by $\phi(\bx) = \bx''$.
The range of $\phi$ is invariant under the group $\CG\type{d-1}$ and 
thus by  Lemma \ref{lemma:Qminimal}, it is equal to $\CQ\type {d-1}$.
By distality, the map $\phi$ is open. 

Assume $(x, x, \ldots, x,y)\in\CQ\type{d}$.  Write 
$\bv = (x, \ldots, x, y)\in X\type {d-1}$.  
Let $\delta > 0$. 
Since $(x\type{d-1}, x\type{d-1})\in \K\type{d}(x)$, by the openness of $\phi$, there exists 
$\delta'$ with $0< \delta' <\delta$ 
such that if $\bu\in\CQ\type{d-1}$ is $\delta'$-close to $x\type{d-1}$, there exists 
$\bz$ that is $\delta$-close to $x\type{d-1}$ and $(\bz, \bu)\in 
\K\type{d}(x)$.  

Since $(x\type{d-1}, \bv)\in\CQ\type{d}$, there exists $\bu\in 
\CQ\type{d-1}$ and $n\in\Z$ such that $\bu$ is at most distance 
$\delta'$ from
$x\type{d-1}$ and $(T\type {d-1})^n\bu$ is at most distance $\delta$ from 
$\bv$.  Taking 
$\bz$ as above, we have that 
$(\bz, (T\type{d-1})^n\bu)\in \K\type{d}(x)$ and is 
$\delta$-close to $(x\type{d-1}, \bv)$.  

Letting $\delta$ go to $0$, we have that $(x\type{d-1}, \bv)\in 
\K\type{d}(x)$.  
\end{proof}

The next result follows directly from Proposition \ref {prop:important} and
the definition of $\QQ\type d$.
It shows that $\QQ\type d$ verifies properties 
that are generalizations of the
$2-$and $3-$dimensional parallelepiped structures as defined 
in~\cite{HM}.  In particular, $\QQ\type d$ satisfies 
the ``property of closing parallelepipeds''.
This plays a key role in our study of the first condition in Theorem~\ref{theorem:main}.

\begin{proposition}
\label{prop:closing}
Let $(X,T)$ be a minimal distal system and let $d \geq 1$ be an integer.
Assume that $x_\epsilon$, $\epsilon\subset [d]$ with $\epsilon \neq [d]$, 
are points in $X$ such that the face
$(x_\epsilon \colon    j \notin \epsilon)$ 
belongs to $\QQ\type{d-1}$ for each  $j \in [d]$.  Then there exists 
$x_{[d]}\in X$ such that $(x_\epsilon \colon\epsilon \subset [d]) \in\QQ\type d$.  
\end{proposition}

Although we have given the last coordinate in the statement of 
this proposition a particular role, using Euclidean permutations 
the analogous statement holds for any other fixed 
coordinate, provided that the corresponding faces 
lie in $\QQ\type{d-1}$.  

\subsection{Strong form of the regionally proximal relation}

\begin{corollary}
\label{cor:replace}
Let $(X,T)$ be a minimal distal system and let $d \geq 1$ be an integer.
Let $x,y\in X$ and $\bb_*\in X\type {d+1}_*$ with 
$(x,\bb_*)\in\QQ\type{d+1}$.
Then $(y,\bb_*)\in\QQ\type{d+1}$ if and only if 
$(y,x,x,\ldots,x)\in\QQ\type{d+1}$.
\end{corollary}
\begin{proof}
We write $\bu=(x,\bb_*)$, $\bv=(y,\bb_*)$,  and $\by=(y,x,x,\ldots,x)\in X\type{d+1}$.
By Proposition~\ref{prop:important}, we have that $\bu$ belongs to $\K\type{d+1}(x)$ and, by minimality, 
there exists a 
sequence $(S_n\colon n\geq 1)$ in $\CF\type{d+1}$ such that 
$S_n\bu\to x\type{d+1}$. Then $S_n\bv\to \by$ and $\by$ belongs to the closed 
orbit of $\bv$ under $\CF\type{d+1}$.   By distality, this last property implies that
$\bv$ belongs to the closed orbit of $\by$.   Since $\QQ\type{d+1}$ is 
closed and invariant under $\CF\type{d+1}$, we have that 
$\by\in\QQ\type{d+1}$ if and only if $\bv\in\QQ\type{d+1}$.
\end{proof}

\begin{corollary}
\label{cor:RPRPs}
Let $(X,T)$ be a minimal distal system and let $d \geq 1$ be an integer.
Let $x,y\in X$. Then $(x,y)\in\RP\type d$ if and only if 
$(y,x,x,\ldots,x)\in\QQ\type{d+1}=\K\type{d+1}(y)$.
\end{corollary}

\begin{proof}
For $\ba^*\in X\type d_*$, apply the preceding corollary 
with $\bb_*=(\ba_*,x,\ba_*)$ and use Lemma~\ref{lemma:RPandQQ}.
\end{proof}

The combination of the previous corollaries allows to prove that each coordinate in a parallelepiped of 
$\CQ\type d$ can be replaced by another point that is regionally proximal of order $d$ with it and
the resulting point is still a parallelepiped. 

We finish with a comment about the regionally proximal relation of order $d$.
In~\cite{Aus}, Corollary 10, chapter 9, 
Auslander  (see also Ellis~\cite{Ellis}) 
proves  that in the definition of the regionally proximal relation, 
the point $x'$ (see Definition~\ref{def:RP} with $d=1$) can be taken to be $x$.
The same result can be stated for the regionally proximal relation of order $d$ in the distal case. 
In fact, a similar argument to the one used to prove Lemma~\ref{lemma:RPandQQ}
allows us to show that: 
$(x,y,\ldots,y)\in\QQ\type{d+1}=\K\type{d+1}(x)$
if and only if  
for any $\delta >0$ there exist $y' \in X$ and a 
vector $\bn=(n_1,\ldots,n_d) \in \Z^d$
such that for any nonempty   
$\epsilon \subset [d]$
$$d(y,y')< \delta \ , \ d(T^{ \bn \cdot \epsilon }x,y)<\delta, \text{ and } d(T^{ \bn \cdot \epsilon}y',y)< \delta$$

\subsection{Summarizing}

\subsubsection{} We show that the second property in Theorem~\ref{theorem:main} implies 
the first one.  Assume that the transitive system $(X,T)$ satisfies 
the second property. By Corollary~\ref{cor:two_distal}, the system is 
distal. 

If $\bx,\by\in\QQ\type{d+1}$ agree on all coordinates other than 
the coordinate indexed 
by $\emptyset$, then $\bx=\by$ by Corollary~\ref{cor:replace}. By 
permutation of coordinates we deduce that the first property of 
Theorem~\ref{theorem:main} is satisfied.

The first two properties of this theorem are thus equivalent. 
From the above discussion, Proposition~\ref{prop:orderRP} follows:
these properties mean that the relation $\RP\type d$ 
is trivial.

\subsubsection{}
\begin{proposition}
\label{prop:quaotient1}
Let $(X,T)$ be a minimal distal system and let $d\geq 1$ be an integer.
Then the relation $\RP\type d$ is a closed invariant equivalence 
relation 
on  $X$.

The quotient of $X$ 
under this equivalence relation is the maximal factor 
of order $d$ of $X$.
\end{proposition}
The second statement means that this quotient is a system of order $d$ and that 
every system of order $d$ which is a factor of $X$ is a factor of 
this quotient.

\begin{proof}
In order to prove the first statement, we are left with showing that
the relation is transitive. Assume that 
$(x,y)$ and $(y,z)\in\RP\type d$. By Corollary~\ref{cor:RPRPs} 
applied to the pair $(x,y)$,
$(y,x,x,\ldots,x)\in\QQ\type {d+1}$.  By Corollary~\ref{cor:replace} 
applied to the pair $(y,z)$, $(z,x,x,\ldots,x)\in\QQ\type{d+1}$ and by 
Corollary~\ref{cor:RPRPs} again, $(x,z)\in\RP\type d$.

We show now the second part of the proposition.
Let $Y$ be the quotient of $X$ under the equivalence relation 
$\RP\type d$ and let $\phi$ denote the factor map.  Let $(a, 
b)\in\RP\type{d}(Y)$.  Then $(a, b, b, \ldots, 
b)\in\CQ\type{d+1}(Y)$. By Lemma 3.1, there exists $\bx\in 
\CQ\type{d+1}(X)$ satisfying $\phi\type{d+1}(\bx) = (a, b, b, \ldots, 
b)$.  

Write $x_\emptyset = x$ and $x_{00\ldots 01} =y$.  For every 
$\epsilon\neq\emptyset$, $\phi(x_\epsilon) = b = \phi(y)$.  Thus 
$(x_\epsilon, y)\in\RP\type{d}(X)$.  Using Corollary~\ref{cor:RPRPs} 
and Corollary~\ref{cor:replace}, we can 
replace $x_\epsilon$ by $y$ in $\bx$ and obtain an element of 
$\CQ\type{d+1}(X)$.  Doing this for all $\epsilon\neq\emptyset$, we 
have that $(x, y, y,\ldots, y)\in\CQ\type{d+1}(X)$.  
By Corollary~\ref{cor:RPRPs}, this means that $(x,y)\in\RP\type{d}(X)$.  Thus that $\phi(x) 
= \phi(y)$ and so $a=b$.  

Let $W$ be a system of order $d$ and let $\psi\colon X\to W$ be a 
factor map.  Take $Y$ and $\phi$ as above and let $x,y\in X$.  If 
$\phi(x) = \phi(y)$, then $(x,y)\in\RP\type{d}(X)$.  Thus by 
Corollary~\ref{lemma:stupid}, $(\psi(x), \psi(y))\in\RP\type{d}(W)$ and thus $\psi(x) = 
\psi(y)$.
\end{proof}

\subsubsection{}

In order to complete the proofs of Theorems~\ref{theorem:main} 
and~\ref{theorem:maxd}, we are left with showing that the notions of a system of order $d$ and 
an inverse limit of $d$-step minimal nilsystems are equivalent.

In one direction, a result from Appendix B of~\cite{HK3}, translated 
into our current vocabulary, states that a $d$-step minimal nilsystem 
is a system of order $d$.  This property easily passes to inverse 
limits, and so we have:
\begin{proposition}\label{prop:Qnil} 
Let $(X,T)$ be an inverse limit of minimal 
$(d-1)$-step nilsystems and let $d\geq 2$ be an integer.  
Then $(X,T)$ is a system of order $d-1$. 
\end{proposition}

 We are left with showing the converse, which is:
\begin{theorem}
\label{th:converse}
Assume that $(X,T)$ is a transitive system of order $d-1$.  Then it is 
an inverse limit of $d-1$-step minimal nilsystems.  
\end{theorem}

We recall that the hypothesis of this theorem means that if $\bx,\by\in\QQ\type {d}$ have $2^{d}-1$ 
coordinates in common, then $\bx=\by$.   In particular, this implies 
that the system is distal and minimal.

The proof of this theorem is carried out in the next two sections.

\section{Ergodic preliminaries}
\label{section:preergodic}

The result  of Theorem~\ref{th:converse}
is established in the next section using invariant 
measures on $X$. In this section, we summarize the background 
material and give some preliminary results.

\subsection{Inverse limits of nilsystems}
\label{subsec:invelimit}

A {\em measure preserving system} is defined to be a 
quadruple $(X, \CB, \mu, T)$, where $(X, \CB, \mu)$ is a 
probability space and $T\colon X\to X$ is a measure 
preserving transformation.  In general, we omit the 
$\sigma$-algebra $\CB$ from the notation and 
write $(X,\mu, T)$.

Throughout, we make use both of the vocabulary of topological 
dynamics and of ergodic theory, leading to possible 
confusion. In general, it is clear from the context whether we are 
referring to a measure preserving system or a topological 
system, and so we just refer to either as a {\em system}.
Topological factor maps were already defined. We recall that an
\emph{ergodic theoretic factor map} between the measure preserving 
systems $(X,  \mu, T)$ and $(X',  \mu', T)$ is a measurable map 
$\pi\colon X\to X'$ (defined almost everywhere), mapping the measure $\mu$ to $\mu'$ and 
commuting with the transformations (almost everywhere).   If the map $\pi$ is 
invertible (almost everywhere), we say that the two systems are {\em isomorphic}.  

Inverse limits of nilsystems in the topological sense were discussed
in Section~\ref{sec:nilsystems}.  We make this notion precise 
in the measure theoretic sense, in this case also we consider only sequential inverse limits.
A $d$-step nilsystem $(X,T)$, endowed with its Haar 
measure $\mu$, is ergodic if and only if $(X,T)$ is a minimal 
topological system; in this case, $\mu$ is its unique 
invariant measure. 
Therefore, every inverse limit 
(in the topological sense) of $d$-step 
minimal nilsystems is uniquely ergodic. 

Now, let $(X,\mu,T)=\varprojlim (X_j,\mu_j,T_j)$ be an inverse limit in 
the ergodic theoretic sense of a sequence of $d$-step ergodic 
nilsystems. Recall that each nilsystem $(X_j,T)$ is endowed with its Borel $\sigma-$algebra
and $\mu_j$ its Haar measure. 
This means that for every $j\in\N$, there exist 
ergodic theoretic factor maps 
$\pi_j\colon (X_{j+1},\mu_{j+1},T)\to(X_j,\mu_j,T)$ and $p_j\colon 
(X,\mu,T)\to(X_j,\mu_j,T)$ satisfying $\pi_j\circ p_{j+1}=p_j$ for 
every $j$ such that the Borel $\sigma$-algebra $\CB$ of $X$ is spanned by 
the $\sigma$-algebras $p_j\inv(\CB_j)$, where $\CB_j$ denotes the 
Borel $\sigma$-algebra of $X_j$. 

Every ergodic theoretic factor map between ergodic nilsystems 
is equal almost everywhere to a topological factor map. A short proof of this fact 
is given in the Appendix (Theorem~\ref{fact:factonil}). 
Therefore, the factor maps $\pi_j$ in the definition of an 
inverse limit (in the ergodic sense) can be 
assumed to be topological factor maps.  
It follows that $(X,\mu,T)$ can 
be identified with the topological inverse limit. 

This allows us, in the sequel, to not distinguish between the notions of 
topological and ergodic theoretic inverse limits of $d$-step 
ergodic nilsystems. 

\subsection{Ergodic uniformity seminorms and nilsystems}
\label{subsec:mud}
Let  $(X,\mu,T)$ be  an ergodic system.
For points in $X\type d$ and transformations of these spaces we use the same notation 
as in the topological setting. 
In Section 3 of~\cite{HK1}, a measure $\mu\type d$ on $X\type d$
and a seminorm $\nnorm\cdot_d$ on $L^\infty(\mu)$ are constructed.

We recall the properties of these objects:
\begin{proposition}
Assume $(X, \mu, T)$ is an ergodic system and that $d\geq 1$ is an 
integer.  
\begin{enumerate}
\item
The measure $\mu\type d$ 
is invariant and ergodic  under the action of the group $\CG \type d$.
\item
Each one dimensional marginal of 
$\mu\type d$ is equal to $\mu$ and each of its two dimensional 
marginals (meaning the image under the map 
$\bx\mapsto(x_\epsilon,x_\theta)$ for $\epsilon\neq\theta\subset[d]$)
is equal to $\mu\times\mu$.
\item
If $p\colon(X,\mu,T)\to(Y,\nu,T)$ is a factor map then, $\nu\type d$ 
is the image of $\mu\type d$ under the map $p\type d\colon X\type 
d\to Y\type d$.
\end{enumerate}
\end{proposition}

For every $f\in L^\infty(\mu)$, the 
\emph{$d$-th seminorm $\nnorm f_d$ of $f$} is defined by
\begin{equation}
\label{eq:seminorm}
 \nnorm 
f_d^{2^d}=\int\prod_{\epsilon\subset[d]}f(x_\epsilon)\,d\mu\type 
d(\bx)\ .
\end{equation}
We have that:
\begin{lemma}
\label{fact:seminorm}
Assume that $(X, \mu, T)$ is an ergodic system and let $d\geq 1$ 
be an integer.
\begin{enumerate}
\item 
\label{it:seminormintegral}
For every $f\in L^\infty(\mu)$, 
$\Bigl|\int f\,d\mu\Bigr|\leq\nnorm f_d$.
\item
If $p\colon(X,\mu,T)\to(Y,\nu,T)$ is a factor map, then 
$\nnorm f_d=\nnorm{f\circ p}_d$ for every
function $f\in L^\infty(\nu)$. 
\end{enumerate}
\end{lemma}

We summarize some of the main results of~\cite{HK1}:
\begin{theorem}
\label{fact:structure}
Assume that $(X,\mu,T)$ is an ergodic system 
and that $d\geq 1$ is an integer.  The following properties are 
equivalent:
\begin{enumerate}
\item
$(X,\mu,T)$ is measure theoretically isomorphic to an inverse limit 
of $(d-1)$-step ergodic  nilsystems.
\item
$\nnorm\cdot_d$ is a norm 
on $L^\infty(\mu)$ (equivalently $\nnorm f_d=0$ implies that $f=0$).
\item
There exists a measurable map $J\colon X\type d_*\to X$ such that 
$x_\emptyset=J(x_\epsilon\colon \emptyset\neq\epsilon\subset[d])$ for 
$\mu\type d$-almost every $\bx\in X\type d$.
\end{enumerate}
\end{theorem}

Using these properties, it follows that any measure theoretic factor of 
an inverse limit 
of $(d-1)$-step nilsystems is isomorphic in the ergodic theoretic 
sense to an inverse limit 
of $(d-1)$-step nilsystems.

\begin{theorem}[\cite{HK1}, Theorem 1.2]
\label{fact:cubiclimit}
Assume that $(X,\mu, T)$ is an ergodic system, $d\geq 1$ is an integer, 
and $f_\epsilon\in L^\infty(\mu)$ for $\emptyset\neq\epsilon\subset[d]$.
The averages
\begin{equation}
\label{eq:cubiclimit}
\frac 1{N^d}\sum_{0\leq n_1,\dots,n_d<N}
\prod_{\substack{\epsilon\subset[d]\\ \epsilon\neq\emptyset}}
f_\epsilon(T^{\bn\cdot\epsilon }x)
\end{equation}
converge in $L^2(\mu)$ as $N\to+\infty$.

Letting $F$ denote the limit  of 
these averages, we have that for every $g\in L^\infty(\mu)$,
\begin{equation}
\label{eq:limcubiclimit}
\int g(x)F(x)\,d\mu(x)=\int g(x_\emptyset)
\prod_{\substack{\epsilon\subset[d]\\ \epsilon\neq\emptyset}}
f_\epsilon(x_\epsilon)\,d\mu\type d(\bx)\ .
\end{equation}
\end{theorem}

\begin{lemma}
\label{fact:boudlim}
Let $(X, \mu, T)$, $d$, 
$f_\epsilon$, $\emptyset\neq\epsilon\subset[d]$, and $F$ be as in 
Theorem~\ref{fact:cubiclimit}.  Then
$$
\norm F_{L^\infty(\mu)}\leq
\prod_{\substack{\epsilon\subset[d]\\ \epsilon\neq\emptyset}}
\norm{f_\epsilon}_{L^{2^d-1}(\mu)}\ .
$$
\end{lemma}

\begin{proof}
Let $g\in L^\infty(\mu)$ and choose a function $h$ with $h^{2^d-1}=g$.
By~\eqref{eq:limcubiclimit} and the H\"older Inequality,
$$
\bigl|\int gF\,d\mu\bigr|\leq
\Bigl(\prod_{\substack{\epsilon\subset[d]\\ \epsilon\neq\emptyset}}
\int |h(x_\emptyset)f_\epsilon(x_\epsilon)|^{2^d-1}\,d\mu\type d(\bx)
\Bigr)^{1/2^d-1}\ .
$$
Since each two dimensional marginal of $\mu\type d$ is equal to 
$\mu\times\mu$, this can be rewritten as 
$$
\Bigl(\prod_{\substack{\epsilon\subset[d]\\ \epsilon\neq\emptyset}}
\int |h(x)f_\epsilon(y)|^{2^d-1}\,d\mu(x)\,d\mu(y)\Bigr)^{1/2^d-1}
=\norm g_{L^1(\mu)}
\prod_{\substack{\epsilon\subset[d]\\ \epsilon\neq\emptyset}}
\norm{f_\epsilon}_{L^{2^d-1}(\mu)}
$$
and the result follows.
\end{proof}

\subsection{Dual functions}
Here again, $(X,\mu,T)$ is an ergodic system.
Following the notation and terminology of~\cite{HK3}, for 
 every $f\in L^\infty(\mu)$, the  limit function
\begin{equation}
\label{eq:defdual}
\lim_{N\to+\infty}\frac 1{N^d}\sum_{n_1,\dots,n_d=0}^{N-1}
\prod_{\emptyset\neq \epsilon \subset d 
}f(T^{\bn \cdot \epsilon}x)
\end{equation}
is called the \emph{dual function of order $d$ of $f$} and is written 
$\CD_df$.
It is worth noting that $\CD_df$ is only defined as an element of 
$L^2(\mu)$, and thus is defined almost everywhere.

By~\eqref{eq:limcubiclimit} and~\eqref{eq:seminorm}, for every $f\in 
L^\infty(\mu)$ we have that
\begin{equation}
\label{eq:duality}
 \int f\,\CD_df\,d\mu=\nnorm f_d^{2^d}\ .
\end{equation}
It follows from Lemma~\ref{fact:boudlim} that:
\begin{lemma} 
\label{fact:boundedDd}
If $(X, \mu, T)$ is an ergodic  system and $d\geq 1$ is an integer, then
for every $f\in L ^\infty(\mu)$:
$$
\norm{ \CD_d f}_{L^\infty(\mu)}\leq \norm f_{L^{2^d-1}(\mu)}^{2^d-1}\ .
$$
Moreover, the map $\CD_d$ extends to a 
continuous map from $L^{2^d-1}(\mu)$ to $L^\infty(\mu)$. 
\end{lemma}

We remark that if $0\leq f\leq g$, then  $0\leq\CD_d f \leq\CD_d g$.

\begin{lemma}
\label{fact:positive}
If $(X, \mu, T)$ is an ergodic system and $d\geq 1$ is an integer, 
then for every $A\subset X$ we have 
$\CD_d\one_A(x)>0$ for $\mu$-almost every $x\in A$.
\end{lemma}

\begin{proof}
Let $B=\{x\in A\colon \CD_d\one_A(x)=0\}$. 
By part~\eqref{it:seminormintegral} of Lemma~\ref{fact:seminorm} 
and~\eqref{eq:duality}, since $\CD_d\one_B\leq\CD_d\one_A$ we have that
$$
\mu(B)^{2^d}\leq \nnorm{\one_B}_d^{2^d}
=
\int_B\CD_d\one_B(x)\,d\mu(x)\leq 
\int_B\CD_d\one_A(x)\,d\mu(x)=0\ .
$$
Thus $\mu(B)=0$.
\end{proof}

Using the definition~\eqref{eq:defdual} of the dual function, we 
immediately deduce:
\begin{lemma}
\label{fact:dualfactor} Let
$p\colon (X,\mu,T)\to (X',\mu',T)$ be a 
measure theoretic factor map. For every $f\in L^\infty(\mu')$
we have
 $(\CD_df)\circ p=\CD_d(f\circ p)$.
\end{lemma}

By Theorem~\ref{fact:cubiclimit}, it follows that:
\begin{lemma}
\label{fact:suppotymud}
Let $(X,T)$ be a minimal topological dynamical system and $\mu$ be
an invariant ergodic measure on $X$. Then the measure $\mu\type d$ is 
concentrated on the subset $\QQ\type d$ of $X\type d$.
\end{lemma}

\begin{lemma}
\label{fact:Ddzero}
Let $(X,T)$ be a minimal system of order $d-1$ and 
let $\mu$ be an invariant ergodic measure on $X$. Let $d_X$ denote 
a distance on $X$ defining the topology of this space and for every 
$x\in X$ and $r>0$, let $B(x,r)$ denote the ball centered at $x$ 
of radius $r$ with respect to the distance $d_X$. 
Then for every $\eta>0$, there exists 
$\delta > 0$ 
such that for every $x\in X$,  $\CD_d\one_{B(x,\delta)}=0$  $\mu$-almost 
everywhere on the complement of $B(x,\eta)$
\end{lemma}

\begin{proof}
By definition of a system of order $d-1$, the last coordinate of an 
element of $\QQ\type d$ is a function of the other ones. Using the 
symmetries of $\QQ\type d$, we have that the same property holds with the first 
coordinate substituted for the last one. 
Therefore, writting $\QQ\type d_*$ for 
$\QQ\type d$ without the first coordinate, 
 there exists a map 
$J\colon \QQ\type d_*\to X$ 
such 
that  for every  $\bx\in\QQ\type d$, 
$$
x_\emptyset=J(x_\epsilon\colon \epsilon\subset[d],\ 
\epsilon\neq\emptyset)\ .
$$
The graph of this map is the closed subset $\QQ\type d$ of 
$\QQ\type d_*\times X$ and thus is continuous.

Fix $\eta>0$. Since $J$ is uniformly continuous and satisfies 
$J(x,\dots,x)=x$ for every $x$, there exists 
$\delta>0$ such that for every $x\in X$, the set
$$
(X\setminus B(x,\eta))\times B(x,\delta)\times\dots\times B(x,\delta)
$$
has empty intersection with $\QQ\type d$.  Thus by Lemma \ref{fact:suppotymud} it has zero $\mu\type 
d$-measure. By Theorem~\ref{fact:cubiclimit} and the definition of 
$\CD_d\one_B$, we have that 
$$
 \int\one_{X\setminus B(x,\eta)}\CD_d\one_{B(x,\delta)}\,d\mu=0\ .
\qed
$$
\renewcommand{\qed}{}
\end{proof}

\subsection{Systems with continuous dual functions}
It is convenient to give a name to the following, although we only make
use of it within proofs:
\begin{definition}
\label{def:propdual}
Let $(X,T)$ be a minimal system and let $\mu$ an ergodic invariant 
measure on $X$. We say that $(X,T,\mu)$ has property $\prop$ if 
whenever $f_\epsilon$, $\emptyset\neq\epsilon\subset[d]$, are 
continuous functions on $X$, the averages~\eqref{eq:cubiclimit} 
converge everywhere and uniformly.
\end{definition}

If this property holds, then in particular, for every continuous 
function $f$ on $X$,
the averages ~\eqref{eq:defdual} converge everywhere and uniformly
 for every continuous function $f$ on $X$. The limit of these averages 
coincides almost everywhere with the function $\CD_df$ defined above
and so we also denote it by $\CD_df$.

\begin{proposition}
\label{fact:uniform}
Let $(X,T)$ be an inverse limit of minimal $(d-1)$-step nilsystems 
and let $\mu$ be the invariant measure of this system. Then $(X,\mu,T)$ has 
property $\prop$.
\end{proposition}

\begin{proof}
Assume first that $(X,\mu,T)$ is a $(d-1)$-step  ergodic nilsystem. 
In~\cite{HK3} (Corollary 5.2), the convergence 
of the averages~\eqref{eq:cubiclimit} is shown to 
hold everywhere and this convergence is uniform when
the functions $f_\epsilon$,
$\emptyset\neq\epsilon\subset [d]$, are continuous.

Assume now that $(X,\mu,T)$ is as in the statement.
Every continuous function on $X$ can be approximated uniformly by  
a continuous function arising from one of the nilsystems which are 
factors of $X$. By density, the result also holds 
in this case.
\end{proof}

We now establish some properties of systems with property $\prop$.
We write $\CC(X)$ for the algebra of continuous functions on $X$. We 
always assume that $\CC(X)$ is endowed with the norm of  uniform 
convergence.

By Lemma~\ref{fact:boundedDd} and density:
\begin{lemma}
\label{fact:continuousDd}
Assume that the ergodic system $(X,\mu,T)$ has property $\prop$.
\begin{itemize}
\item
For every $f\in L^{2^d-1}(\mu)$, the 
function $\CD_df$ is equal $\mu$-almost everywhere to a continuous 
function on $X$,  which we also denote by $\CD_df$, called
the dual function of $f$.
\item
The map $f\mapsto\CD_df$ is continuous from $L^{2^d-1}(\mu)$ to
$\CC(X)$.
\end{itemize}
\end{lemma}

\begin{lemma}
\label{fact:factoPd}
Let $(X,\mu,T)$ be a system with property $\prop$, $(Y,T)$ be a 
minimal system, $p\colon X\to Y$ a topological factor map, and $\nu$ be
the image of $\mu$ under $p$.
Then $(Y,T,\nu)$ has property $\prop$.
\end{lemma}

\begin{proof}
Let
$f_\epsilon$, $\emptyset\neq\epsilon\subset[d]$, be
continuous functions on $Y$. Then the averages
$$
 \frac 1{N^d}\sum_{0\leq n_1,\dots,n_d<N}
\prod_{\substack{\epsilon\subset[d]\\ \epsilon\neq\emptyset}}
f_\epsilon(T^{\bn\cdot\epsilon }p(x))
$$
converge uniformly on $X$ and thus  the averages~\eqref{eq:cubiclimit} 
converge uniformly on $Y$.
\end{proof}

\section{Using a measure}
\label{sec:final}
In this section, we prove Theorem~\ref{th:converse} which completes the proof of 
Theorem~\ref{theorem:main}: any transitive system 
$(X,T)$ of order $d-1$ is an inverse limit of $(d-1)$-step minimal 
nilsystems. By  Corollary~\ref{cor:two_distal}, $(X,T)$ is distal 
and thus is minimal.
The method we use is completely different from that used in~\cite{HM} 
for $d=3$, and 
proceeds by introducing an invariant measure on $X$.

We start by reducing the proof of Theorem~\ref{th:converse} to 
the following:
\begin{proposition}
\label{prop:iso}
Let $(X,T)$ be a minimal system of order $d-1$, $\mu$ be an invariant 
ergodic measure on $X$, and let $(Y,T)$ be an inverse limit of 
minimal $(d-1)$-step nilsystems with Haar measure $\nu$.  
Let $\Psi\colon (Y,\nu,T)\to(X,\mu,T)$ be a measure theoretic 
isomorphism.
Then $\Psi$ coincides $\nu$-almost everywhere with a topological 
isomorphism.
\end{proposition}

\begin{proof}[Proof of Theorem~\ref{th:converse} (assuming 
Proposition~\ref{prop:iso})]
By Lemma~\ref{fact:suppotymud}, the measure $\mu\type d$ is 
concentrated on the subset $\QQ\type d$ of $X\type d$.
Since $(X,T)$ is a system of order $d-1$, there exists a continuous map
$J\colon \QQ\type d_*\to X$ such 
that  
$$
x_\emptyset=J(x_\epsilon\colon \epsilon\subset[d],\ 
\epsilon\neq\emptyset)\text{ for every  }\bx\in\QQ\type d
$$
and so this property holds $\mu\type d$-almost everywhere.  (Again, 
$\QQ\type d_{*}$ denotes $\QQ\type d$ without the first coordinate.)  
By Theorem~\ref{fact:structure}, $(X,\mu,T)$ is isomorphic in the 
ergodic theoretic sense to an inverse limit $(Y,\nu,T)$ of 
$(d-1)$-step ergodic nilsystems.
By Proposition~\ref{prop:iso}, $(X,T)$ and $(Y,S)$ are isomorphic in 
the topological sense and we are finished.
\end{proof}

\subsection{Proof of Proposition~\ref{prop:iso}}
\label{subsec:endrpoof}

To prove Proposition~\ref{prop:iso}, we start with a lemma:
\begin{lemma}
\label{lem:Phicontinuous}
Let $(Y,\nu,T)$ be a system with Property $\prop$, $(X,T)$ be a minimal 
system of order $d-1$, $\mu$ be an invariant probability measure on $X$, 
and $\Psi\colon (Y,\nu,T)\to(X,\mu,T)$ be a measure theoretic factor 
map.  Then $\Psi$ agrees $\nu$-almost everywhere with some topological 
factor map.
\end{lemma}
\begin{proof}

We can assume that there exists a Borel invariant subset $Y_0$ of full 
measure and  that 
$\Psi$ is a Borel map from $Y_0$ to $X$, mapping the measure $\nu$ to 
the measure $\mu$ and such that  $\Psi(Tx)=T\Psi(x)$ for 
every $x\in Y_0$.

We claim that:
\begin{claim}
\label{fact:openopen}
For every open subset $U$ of $X$, there exists an open subset $\widetilde 
U$ of $Y$ equal to $\Psi\inv(U)$ up to a $\nu$-negligible set.
\end{claim}

To see  this, if $(X,\mu)$ is a probability space and $A,B\subset X$, write
$A\subset_\mu B$ if $\mu(A\setminus B)=0$. The notations $A\supset_\mu 
B$ and $A=_\mu B$ are defined similarly.

Assume that $U\neq\emptyset$, as
otherwise the claim holds trivially. Let $x\in U$.  
By Lemma~\ref{fact:Ddzero} there exists an 
open subset $W_x$ containing $x$ and included in $U$ such that  the 
set
$$
 U_x:=\{x\in X\colon\CD_d\one_{W_x}>0\}
$$
satisfies 
\begin{equation}
\label{eq:UxU}
 U_x\subset_\mu U\ .
\end{equation}

Define
$$
 \widetilde U_x= \bigl\{y\in Y\colon \CD_d(\one_{W_x}\circ\Psi)(y)>0\bigr\}\  .
$$
By Lemmas~\ref{fact:continuousDd} and~\ref{fact:dualfactor}, $\widetilde U_x$ is
an open subset of $Y$ and 
$$
 \widetilde U_x=_\nu \Psi\inv(U_x)\ .
$$

We have that $U$ is the union of the open sets $W_x$ for $x\in U$.  Since 
$U$ is $\sigma$-compact, there exists a countable subset $\Gamma$ of $U$ 
such that the union $\bigcup_{x\in\Gamma}U_x$ is equal 
to $U$. Define
$$
 \widetilde U=\bigcup_{x\in\Gamma}\widetilde U_x\ .
$$
Then
$$
 \widetilde U=_\nu \Psi\inv\bigl(\bigcup_{x\in\Gamma}U_x\bigr)\ .
$$
By~\eqref{eq:UxU}, $\widetilde U\subset_\nu \Psi\inv(U)$.
By Lemma~\ref{fact:positive}, for every $x\in\Gamma$ we have that 
$ U_x\supset_\mu W_x$.  Thus $\widetilde U_x\supset_\nu \Psi\inv(W_x)$ 
and
$$
 \widetilde U\supset_\nu\bigcup_{x\in\Gamma}\Psi\inv(W_x)=
\Psi\inv\bigl(\bigcup_{x\in\Gamma}W_x\bigr)=\Psi\inv(U)\ .
$$
This completes the proof of  the claim.  

\begin{claim}
\label{fact:continuousalmost}
There exists an invariant subset $Y_1$ of full measure such that the 
restriction of $\Psi$ to $Y_1$ 
(endowed with the induced topology) is continuous.
\end{claim}

To prove this claim, we let $(U_j\colon j\geq 1)$ be a countable basis for the topology of $X$.

For every $j\geq 1$, by Claim~\ref{fact:openopen} there exists an open 
subset $\widetilde U_j$ of $Y$ such that 
the symmetric difference
$$
 Z_j:= \widetilde U_j \ \Delta \ \Psi\inv(U_j)
$$
has zero $\nu$-measure. Define
$$
Y_1=Y_0\setminus  \bigcup_{n\in\Z}\bigcup_{j\geq 1} T^{n}Z_j\ .
$$
(Recall that $Y_0$ is the invariant subset of $Y$ where the map $\Phi$ 
is defined.)

For every $j\geq 1$, $\Psi\inv(U_j)\cap Y_1=\widetilde U_j\cap Y_1$. 
Every nonempty open subset $U$ of $X$ is the union of some of the 
sets $U_j$, and if $\widetilde U$ is the union of the corresponding sets 
$\widetilde U_j$ we have that $\Psi\inv(U)\cap Y_1=\widetilde U\cap Y_1$. 
This proves the claim.

We combine these results to complete the proof of 
Lemma~\ref{lem:Phicontinuous}.  Since 
$(Y,T)$ is minimal, the measure $\nu$ has full support in $Y$ and 
the subset $Y_1$ given by Claim~\ref{fact:continuousalmost} is dense 
in $Y$. Since $(X,T)$ is distal, the result now follows from 
Lemma~\ref{lem:extendcontinu}.

\end{proof}

Using this, we return to the proposition:
\begin{proof}[Proof of Proposition~\ref{prop:iso}]
There exist a Borel invariant subset $Y_0$ of $Y$ of full measure, a 
Borel invariant subset $X_0$ of full measure, and a Borel bijection 
$\Psi\colon Y_0\to X_0$ with Borel inverse, mapping $\nu$ to $\mu$ 
and commuting with the transformations.

Recall that $(X,T)$ is a system of order $d-1$ and that $(Y,\nu,S)$ 
satisfies property $\prop$.
By Lemma~\ref{lem:Phicontinuous}, there exist a subset $Y_1$ of $Y_0$ 
of full measure and a topological factor map $\Phi\colon Y\to X$
that coincides with $\Psi$ on $Y_1$.

By Lemma~\ref{fact:factoPd}, $(X,\mu,T)$ has property $\prop$. 
Recall that $(Y,T)$ is a system of order $d-1$. Using 
Lemma~\ref{lem:Phicontinuous} again, there exist a subset $X_1$ of 
$X_0$ of full measure and a topological factor map $\Theta\colon X\to 
Y$ that coincides with $\Psi\inv$ on $X_1$.

The subset $Y_1\cap \Psi\inv(X_1)$ has full measure in $Y$ and for 
$y$ in this set, we have $\Theta\circ\Phi(y)=y$. Since the measure 
$\nu$ has full support, this equality holds everywhere and 
$\Theta\circ\Phi=\id_Y$. 
By the same argument, $\Phi\circ\Theta=\id_X$ and we are done.
\end{proof}

\appendix
\section{Rigidity properties of inverse limits of nilsystems}
\label{sec:inutile}

In this Section, we assume that 
$d>1$ is an integer and establish some 
``rigidity''  properties of inverse limits of $(d-1)$-step 
nilsystems, meaning some continuity properties.
 
A property of nilsystems  of this 
type (Theorem~\ref{fact:factonil}) was 
used in Section~\ref{subsec:invelimit} in the discussion on the 
definition of inverse limits, and so the reader may be concerned 
about a possible vicious circle in the argument.  
The way to avoid this is to first carry out the results 
in this section for nilsystems, and not 
inverse limits of nilsystems.  This suffices to establish 
the property needed in Section~\ref{subsec:invelimit}.
Then it is easy to check that the same proofs 
extend to the general case. 

Throughout the remainder of this section, we assume 
that $(X,T)$ is an inverse limit of minimal 
$(d-1)$-step nilsystems and that 
$\mu$ is the invariant measure of this system.
We recall that $(X,T)$ is a system of order $d-1$ and has property
$\prop$ of continuity of dual functions (Proposition~\ref{fact:uniform}).
We first give a slight improvement of Lemma~\ref{fact:Ddzero}, 
maintaining the same notation:
\begin{lemma}
\label{fact:Ddzero2}
For every $x\in X$ and every neighborhood $U$ of $X$, there exists a 
neighborhood $V$ of $x$ such that
if $f$ is a continuous function on $X$ whose support 
lies in $V$, then the support of the
function $\CD_df$ is contained in $U$.
\end{lemma}

\begin{proof}
Pick $\eta>0$ such that the ball $B(x,2\eta)$ is 
contained in $U$. Let $\delta$ be as in 
Lemma~\ref{fact:Ddzero} and let $V=B(x,\delta)$. 

Assume that $f\in\CC(X)$ has support contained in $V$ and 
assume that $|f|\leq 1$.  
We have that $|\CD_df|\leq\CD_d|f|\leq 
\CD_d\one_{B(x,\delta)}$. 
By the choice of $\delta$, $\CD_df$  is
equal to zero almost 
everywhere on the complement of $B(x,\eta)$.

Since the function $\CD_d f$ is continuous and since the 
measure $\mu$ has full support in $X$, $\CD_df$ 
vanishes everywhere outside the closed ball $\bar B(x,\eta)$, which is 
included in $U$.
\end{proof}

\begin{lemma}
\label{fact:Dfnonzero}
If $f$ is a nonnegative continuous function on $X$,
then $\CD_df(x)>0$ for every $x\in X$ such that $f(x)>0$.
\end{lemma}

\begin{proof}
It follows immediately from property $\prop$ 
that  for every $x\in X$, 
there exists a probability measure 
$\mu\type d_x$ on $X\type d_*$ such that 
$$
\frac 1{N^d}\sum_{0\leq n_1,\dots,n_d<N}
\prod_{\substack{\epsilon\subset[d]\\ \epsilon\neq\emptyset}}
f_\epsilon(T^{\bn\cdot\epsilon }x)
\to
\int \prod_{\substack{\epsilon\subset[d]\\ \epsilon\neq\emptyset}}
f_\epsilon(y_\epsilon)\,d\mu\type d_x(\by_*)
$$
as $N\to+\infty$ for any continuous functions $f_\epsilon$, $\emptyset\neq\epsilon\subset[d]$, on $X$.

By construction, the measure $\delta_x\times\mu\type d_x$ is concentrated on the closed orbit $\K\type d(x)$ of the point 
$x\type d\in X\type d$ under the group of face transformations 
$\CF \type d$, and is invariant under these transformations. 
Since $(X,T)$ is distal, the action of these transformations on 
$\K\type d(x)$ is minimal and thus the topological support of the 
measure $\delta_x\times\mu\type d_x$ is equal to $\K\type d(x)$.
Therefore, the point $x_*\type d\in X\type d_*$ belongs to the 
topological support of the measure $\mu\type d_x$. 

If $f$ is a nonnegative continuous function on $X$  with 
$f(x)>0$ then, 
$$
\CD_df(x)=\int 
\prod_{\substack{\epsilon\subset[d]\\ \epsilon\neq\emptyset}}
f(y_\epsilon)\,d\mu\type d_x(\by_*)>0 \ ,
$$
because the function in the integral is positive at the point 
$x_*\type d$ which belongs to the support of the measure $\mu\type 
d_x$.
\end{proof}

\begin{lemma}
\label{fact:Dddense}
The algebra of functions spanned by $\{\CD_d f\colon f\in 
\CC(X)\}$ is dense in $\CC(X)$ under the uniform norm.
\end{lemma}

\begin{proof} 
By Lemmas~\ref{fact:Ddzero2} and~\ref{fact:Dfnonzero}, for 
distinct $x,y\in X$, there exists a continuous function $f$ on $X$ 
with $\CD_df(x)\neq\CD_df(y)$. Recall that $\CD_df$ is a continuous 
function on $X$.  Noting that $\CD_d1=1$, the 
statement follows from the Stone-Weierstrass Theorem.
\end{proof}

\begin{theorem}
\label{fact:factonil}
Let $p\colon (X,\mu,T)\to(X',\mu',T')$ be a measure theoretic 
factor map between inverse limits of $(d-1)$-step 
ergodic nilsystems. Then the factor map $p\colon X\to X'$ is   
equal almost everywhere to a topological factor map.
\end{theorem} 

\begin{proof}
Let $\CA$ be a countable subset of $\CC(X')$ that is dense under 
the uniform norm. 
By  Lemmas~\ref{fact:continuousDd} and~\ref{fact:Dddense},
$\{\CD_df\colon f\in\CA\}$ is included in $\CC(X')$ and is 
dense in this algebra.

By Lemma~\ref{fact:dualfactor}, 
for every $f\in\CA$ we have that $\CD_df\circ 
p=\CD_d(f\circ p)$ almost everywhere.
By Lemma~\ref{fact:continuousDd}, $\CD_d(f\circ 
p)$ is $\mu$-almost everywhere equal 
to a continuous function on $X$. 
Therefore, there exists $X_0\subset X$ of full measure such 
that for every $f\in\CA$, the function $(\CD_df)\circ p$ coincides on 
$X_0$ with a continuous function on $X$.  
The same property holds for 
every function belonging to the algebra spanned by $\CA$. 
Since $X_0$ is dense in $X$, by density the same property 
holds for every continuous function on $X$. 

This defines a homomorphism of algebras $\kappa\colon 
\CC(X')\to\CC(X)$ with $\kappa f(x)=f(p(x))$ for every $x\in X_0$ and 
every $f\in\CC(X')$, and $\kappa$ commutes with the 
transformations $T$ and $T'$.  
Thus there exists a continuous map $p'\colon X\to X'$ 
such that $\kappa f=f\circ p'$ for all $f\in\CC(X')$. 
\end{proof}

\begin{theorem}
\label{fact:acionnil} Let $(X,T,\mu)$ be an ergodic inverse limit of 
$(d-1)$-step nilsystems, $G$ be a Polish group, 
and $(g,x)\mapsto g\cdot 
x$ be a Borel action of $G$ on $X$ by measure preserving transformations commuting with $T$.
There exists a continuous action $(g,x)\mapsto g*x$ of $G$ on 
$X$, commuting with $T$, such 
that for every $g\in G$, 
$g*x=g\cdot x$ for $\mu$-almost every $x\in X$.
\end{theorem}

By hypothesis, the map $(g,x)\mapsto g\cdot x$ is Borel from 
$G\times X$ to $X$. The action of $G$ on $X$ we want must be such 
that the map $(g,x)\mapsto g*x$  is continuous from $G\times X$ to 
$X$.

\begin{proof}

By Theorem~\ref{fact:factonil}, for every $g\in G$ there exists a 
continuous map $x\mapsto g* x$, commuting with $T$ and preserving 
the measure $\mu$, such that 
$g* x=g\cdot x$ for $\mu$-almost every $x\in X$.
For $g,h\in G$, we have that for $\mu$-almost every $x\in X$, 
$g* (h\cdot x)=gh* x$.  By density, the same equality 
holds everywhere. Therefore, the map $(g,x)\mapsto g* x$ is an 
action of $G$ on $X$.   We are left with showing that this 
map is jointly continuous.
 
Let $f\in\CC(X)$. For $g\in G$, write $f_g(x)=f(g* x)$.
For each $g\in G$, the function $f_g$ is continuous and the 
map $x\mapsto g* x$ commutes with $T$.  
By Proposition~\ref{fact:uniform},  $\CD_df_g(x)=\CD_df(g* x)$ for every $x\in X$.

For each $g\in G$, the functions $f_g$ and $x \mapsto g\cdot x$ are 
equal almost everywhere and represent the same element of 
$L^{2^d-1}(\mu)$.  
Since the action $(g,x)\mapsto g\cdot x$ of $G$ on $X$ 
is Borel and measure preserving, by~\cite{BK} we have that the map 
$g\mapsto f_g$ is continuous from $G$ to $L^{2^d-1}(\mu)$.
By Lemma~\ref{fact:boundedDd}, the map $g\mapsto \CD_df_g$ is continuous from $G$ to $\CC(X)$, meaning that the 
function 
$(g,x)\mapsto\CD_df_g(x)=\CD_df(g* x)$ is continuous on
 $G\times X$. 

By density
(Lemma~\ref{fact:Dddense}), for every function $h\in\CC(X)$,
the function $(g,x)\mapsto h(g*x)$ is continuous on $G\times X$. We deduce that the map $(g,x)\mapsto 
g* x$ is continuous from $G\times X$ to $X$.
\end{proof}

\end{document}